\documentclass[11pt,reqno,tbtags,draft]{amsart}
\usepackage{amssymb}
\usepackage{url}
\usepackage[square,numbers]{natbib}
\bibpunct[, ]{[}{]}{;}{n}{,}{,}

\title[Asymptotics of fluctuations in Crump--Mode--Jagers processes]
{Asymptotics of fluctuations in Crump--Mode--Jagers processes: the lattice
case} 
\date{29 May, 2018}

\author{Svante Janson}
\thanks{Partly supported by the Knut and Alice Wallenberg Foundation}
\address{Department of Mathematics, Uppsala University, PO Box 480,
SE-751~06 Uppsala, Sweden}
\email{svante.janson@math.uu.se}
\newcommand\urladdrx[1]{{\urladdr{\def~{{\tiny$\sim$}}#1}}}
\urladdrx{http://www2.math.uu.se/~svante/}


\numberwithin{equation}{section}

\renewcommand\le{\leqslant}
\renewcommand\ge{\geqslant}

\allowdisplaybreaks

\theoremstyle{plain}
\newtheorem{theorem}{Theorem}[section]

\newtheorem{thm}[theorem]{Theorem}
\newtheorem{lem}[theorem]{Lemma}

\theoremstyle{definition}

\newtheorem{ex}[theorem]{Example}  

\newtheorem{rem}[theorem]{Remark}  

\newtheorem*{ack}{Acknowledgement}

\theoremstyle{remark}

\newenvironment{romenumerate}[1][0pt]{
\addtolength{\leftmargini}{#1}\begin{enumerate}
 }{\end{enumerate}}

\newenvironment{Aenumerate}[1][0pt]{
\addtolength{\leftmargini}{#1}\begin{enumerate}

 }{\end{enumerate}}

\newenvironment{Bnoenumerate}[1][0pt]{
\addtolength{\leftmargini}{#1}\begin{enumerate}

 }{\end{enumerate}}

 \newenvironment{Cnoenumerate}[1][0pt]{
\addtolength{\leftmargini}{#1}\begin{enumerate}

 }{\end{enumerate}}

\newcounter{oldenumi}
{\setcounter{oldenumi}{\value{enumi}}
\begin{romenumerate} \setcounter{enumi}{\value{oldenumi}}}
{\end{romenumerate}}

\newenvironment{Aenumerateq}
{\setcounter{oldenumi}{\value{enumi}}
\begin{Aenumerate} \setcounter{enumi}{\value{oldenumi}}}
{\end{Aenumerate}}

\newcounter{thmenumerate}

\newcounter{xenumerate}   

\newcommand{\refT}[1]{Theorem~\ref{#1}}

\newcommand{\refL}[1]{Lemma~\ref{#1}}
\newcommand{\refR}[1]{Remark~\ref{#1}}
\newcommand{\refS}[1]{Section~\ref{#1}}
\newcommand{\refSS}[1]{Subsection~\ref{#1}}

\newcommand{\refE}[1]{Example~\ref{#1}}

\begingroup
  \divide\count255 by 60
  \multiply\count255 by -60
  \advance\count255 by \time
  \ifnum \count255 < 10 \xdef\klockan{\the\count1.0\the\count255}
  \else\xdef\klockan{\the\count1.\the\count255}\fi
\endgroup

\newcommand{\sumiq}{\sum_{i=1}^q}
\newcommand{\sumpq}{\sum_{p=1}^q}
\newcommand{\sumtq}{\sum_{t=1}^q}
\newcommand{\sumii}{\sum_{i=1}^\infty}
\newcommand{\sumj}{\sum_{j=0}^\infty}
\newcommand{\sumji}{\sum_{j=1}^\infty}
\newcommand{\sumk}{\sum_{k=0}^\infty}
\newcommand{\sumki}{\sum_{k=1}^\infty}
\newcommand{\suml}{\sum_{\ell=0}^\infty}
\newcommand{\sumln}{\sum_{\ell=0}^n}
\newcommand{\sumkn}{\sum_{k=0}^n}
\newcommand{\sumkin}{\sum_{k=1}^n}

\newcommand{\sumiN}{\sum_{i=1}^N}

\newcommand{\sumjM}{\sum_{j=1}^M}
\newcommand{\sumkM}{\sum_{k=0}^M}

\newcommand\set[1]{\ensuremath{\{#1\}}}

\newcommand\Bigset[1]{\ensuremath{\Bigl\{#1\Bigr\}}}

\newcommand\xpar[1]{(#1)}
\newcommand\bigpar[1]{\bigl(#1\bigr)}
\newcommand\Bigpar[1]{\Bigl(#1\Bigr)}
\newcommand\biggpar[1]{\biggl(#1\biggr)}
\newcommand\lrpar[1]{\left(#1\right)}

\newcommand\xcpar[1]{\{#1\}}

\newcommand\bigabs[1]{\bigl|#1\bigr|}
\newcommand\Bigabs[1]{\Bigl|#1\Bigr|}

\newcommand\lrabs[1]{\left|#1\right|}
\def\rompar(#1){\textup(#1\textup)}    
\newcommand\xfrac[2]{#1/#2}

\newcommand\xqfrac[2]{#1/(#2)}

\newcommand\Bigparfrac[2]{\Bigpar{\frac{#1}{#2}}}

\newcommand\innprod[1]{\langle#1\rangle}
\newcommand\biginnprod[1]{\bigl\langle#1\bigr\rangle}
\newcommand\Biginnprod[1]{\Bigl\langle#1\Bigr\rangle}

\def\xexp(#1){e^{#1}}

\newcommand\floor[1]{\lfloor#1\rfloor}

\newcommand\ntoo{\ensuremath{{n\to\infty}}}
\newcommand\Mtoo{\ensuremath{{M\to\infty}}}

\newcommand\ltoo{\ensuremath{{\ell\to\infty}}}

\newcommand\norm[1]{\|#1\|}
\newcommand\bignorm[1]{\bigl\|#1\bigr\|}
\newcommand\Bignorm[1]{\Bigl\|#1\Bigr\|}

\newcommand\punkt{.\spacefactor=1000}    
\newcommand\iid{i.i.d\punkt}    
\newcommand\ie{i.e\punkt}
\newcommand\eg{e.g\punkt}

\newcommand\cf{cf\punkt}
\newcommand{\as}{a.s\punkt}

\newcommand\ii{\mathrm{i}}

\newcommand{\tend}{\longrightarrow}
\newcommand\dto{\overset{\mathrm{d}}{\tend}}
\newcommand\pto{\overset{\mathrm{p}}{\tend}}
\newcommand\asto{\overset{\mathrm{a.s.}}{\tend}}
\newcommand\eqd{\overset{\mathrm{d}}{=}}

\newcommand\Oas{O_{\mathrm a.s.}}

\newcommand\bbR{\mathbb R}
\newcommand\bbC{\mathbb C}

\newcommand\bbZ{\mathbb Z}

\newcounter{CC}
\newcommand{\CC}{\stepcounter{CC}\CCx} 
\newcommand{\CCx}{C}  
\newcommand{\CCdef}[1]{\xdef#1{\CCx}}     
\newcounter{cc}
\newcommand{\cc}{\stepcounter{cc}\ccx} 
\newcommand{\ccx}{c_{\arabic{cc}}}     
\newcommand{\ccdef}[1]{\xdef#1{\ccx}}     

\renewcommand\Re{\operatorname{Re}}

\newcommand\E{\operatorname{\mathbb E{}}}
\newcommand\PP{\operatorname{\mathbb P{}}}
\newcommand\Var{\operatorname{Var}}
\newcommand\Cov{\operatorname{Cov}}
\newcommand\Corr{\operatorname{Corr}}

\newcommand\supp{\operatorname{supp}}

\newcommand\ga{\alpha}
\newcommand\gb{\beta}
\newcommand\gd{\delta}
\newcommand\gD{\Delta}
\newcommand\gam{\gamma}
\newcommand\gG{\Gamma}

\newcommand\gl{\lambda}
\newcommand\gL{\Lambda}

\newcommand\gS{\Sigma}
\newcommand\gs{\sigma}
\newcommand\gss{\sigma^2}
\newcommand\gth{\vartheta}
\newcommand\eps{\varepsilon}

\renewcommand\phi{\xxx}  

\newcommand\bz{\bar z}

\newcommand\cB{\mathcal B}

\newcommand\cF{\mathcal F}
\newcommand\cG{\mathcal G}
\newcommand\cH{\mathcal H}
\newcommand\cI{\mathcal I}

\newcommand\cN{\mathcal N}

\newcommand\cR{{\mathcal R}}

\newcommand\cZ{{\mathcal Z}}

\newcommand\ett[1]{\boldsymbol1\xcpar{#1}}

\newcommand\qw{^{-1}}
\newcommand\qww{^{-2}}
\newcommand\qq{^{1/2}}
\newcommand\qqw{^{-1/2}}

\newcommand\intoo{\int_0^\infty}

\newcommand\ooo{[0,\infty)}

\newcommand\dd{\,\mathrm{d}}

\newcommand\lhs{left-hand side}
\newcommand\rhs{right-hand side}

\newcommand\CMJ{Crump--Mode--Jagers}
\newcommand\CMJbp{\CMJ{} branching process}

\newcommand\zgf{Z^\gf}
\newcommand\zgfo{Z^{\chio}}

\newcommand\mmu{\widehat\mu}
\newcommand\mXi{\widehat\Xi}

\newcommand\vecx{\vec X}
\newcommand\vecv{\vec v}
\newcommand\vecu{\vec u}
\newcommand\veca{\vec a}
\newcommand\xoo{_0^\infty}
\newcommand\koo{_{k=0}^\infty}
\newcommand\joo{_{j=0}^\infty}
\newcommand\llr{\ell^2_R}
\newcommand\llrqw{{\ell^2_{R\qw}}}
\newcommand\hhr{H^2_R}

\newcommand\vecN{\vec N}
\newcommand\veceta{\vec\eta}
\newcommand\vecgth{\vec\gth}
\newcommand\Ni[1]{N_{i,#1}}
\newcommand\xii{_i}
\newcommand\kk{^{(k)}}
\newcommand\kkx[1]{^{(#1)}}
\newcommand\kki{^{(k,i)}}
\newcommand\bB{\overline B}

\newcommand\bW{\overline W}
\newcommand\refAA{\ref{Afirst}--\ref{Ar}}
\newcommand\refAAA{\refAA{} and \ref{Aroots}}
\newcommand\refAAAA{\refAA, \ref{Aroots} and \ref{Agf}}
\newcommand\refAroots{B}
\newcommand\ArootsR{(\refAroots${}'$)}

\newcommand\nM{_{n,M}}
\newcommand\dw{\mathrm{d}w}
\newcommand\dz{\mathrm{d}z}
\newcommand\glgf{\gl^{\gf}}
\newcommand\gLgf{\gL^{\gf}}
\newcommand\vgf{V^{\gf}}
\newcommand\vecgdl{\gD\vec\gl^{\gf}}
\newcommand\gamb[1]{\kappa_{#1}}
\newcommand\gGG{\gG_*}
\newcommand\gGGG{\gG_{**}}
\newcommand\ggx{\gam_*}
\newcommand\ddl{\frac{\mathrm d}{\mathrm{d}\gl}}
\newcommand\rhox{r}
\newcommand\SI{\cI}
\newcommand\llll{^{(\ell)}}
\newcommand\bgam{\bar\gamma}
\newcommand\FF{\cF}
\newcommand\GG{\cG}
\newcommand\gDM{\gD M}
\newcommand\gDMx{\gD \Mx}
\newcommand\gDMy{\gD \My}

\newcommand\Mx{\widehat M}
\newcommand\My{\widetilde M}
\newcommand\Vx{\widehat V}
\newcommand\Vy{\widetilde V}
\newcommand\wrt{with respect to}
\newcommand\chix{\Psi}
\newcommand\gf{\chi}
\newcommand\hT{\hat T}
\newcommand\tR{\tilde R}
\newcommand\ellj{j}
\newcommand\ellk{k}
\newcommand\trx{\tilde r}
\newcommand\gsx{\sigma^*}
\newcommand\ettmx{\frac{m-1}m}
\newcommand\qU{\check U}
\newcommand\zchia{\zgfo_n}
\newcommand\zchib{\innprod{\vecx_n,\vecgdl}}
\newcommand\chio{\tilde\chi}

\newcommand{\Polya}{P\'olya}

\begin{document}

%

\begin{abstract}
Consider a supercritical Crump--Mode--Jagers process such that all births
are at integer times (the lattice case).
Let $\mmu(z)$ be the generating function of the intensity of the offspring
process, and consider the complex roots of $\mmu(z)=1$. The smallest (in
absolute value)  such root is $e^{-\ga}$, where $\ga>0$ is the Malthusian
parameter; let $\ggx$ be the second smallest absolute value of a root.

We show, assuming some technical conditions, that there are three cases:
\begin{romenumerate}
\item 
if $\ggx>e^{-\ga/2}$, then
the second-order fluctuations of the age distribution are asymptotically
normal;
\item if $\ggx=e^{-\ga/2}$, then
the fluctuations are still 
asymptotically normal, but with a larger order of the variance; 
\item if $\ggx<e^{-\ga/2}$, then
the fluctuations are even larger,  but will oscillate and
(except in degenerate cases) not
converge in distribution.
\end{romenumerate}
This trichotomy is similar to what has been seen in related situations,
e.g.\ for some other branching processes, and for P\'olya urns.

The results lead to a
symbolic calculus
describing the limits.
The results extends to populations counted by a random characteristic.
\end{abstract}

\maketitle

\section{Introduction}\label{S:intro}

Consider a \CMJbp{},
 starting with a single individual born at time 0, 
where
an individual has $N\le\infty$ children born at the times when the parent
has age
$\xi_1\le\xi_2\le\dots$. 
Here $N$ and $(\xi_i)_i$ are random, and different individuals have
independent copies of these random variables. 
Technically, it is convenient
to regard $\set{\xi_i}_1^N$ as a point process $\Xi$ on $\ooo$, and give
each individual $x$ an independent copy $\Xi_x$ of $\Xi$.
For further details, see \eg{} \citet{Jagers}.

We consider the supercritical case, when the population grows to infinity
(at least with positive probability). As is well-known, under weak
assumptions, the population grows exponentially, like $e^{\ga t}$ for some
constant $\ga>0$ known as the \emph{Malthusian parameter}, see e.g.\ 
\cite[Theorems (6.3.3) and (6.8.1)]{Jagers};
in particular, the population size properly normalized converges to some
positive random variable, and the age distribution stabilizes.
Our purpose is to study the second-order fluctuations of
the age distribution, or more generally, 
of the
population counted with a random characteristic.

We consider in this paper 
the lattice case; we thus assume
that
the $\xi_i$ are integer-valued and thus all
births occur at integer times \as, but there is no $d>1$ such that all birth
times \as{} are divisible by $d$.

Our setting can, for example, 
be considered as a model for the (female)
population of some animal that is fertile several years and gets one or
several children once every year, 
with the numbers of children different years random
and dependent.

Our main results (Theorems \ref{T1}--\ref{T3})
show  that there are three different cases 
depending on properties of the intensity measure
$\E\Xi$ of the offspring process:
in one case fluctuations are, after proper normalization,
asymptotically normal, with only a short-range dependence between different
times;
in another case, there is a long-range dependence and, again after proper
normalization (different this time), 
the fluctuations are \as{} approximated by oscillating (almost periodic)
random functions of
$\log n$, which furthermore essentially are determined by
the initial phase of the branching process, and presumably non-normal;
the third case is an intermediate boundary case.
See \refS{Smain} for precise results.

A similar trichotomy has been found in several related
situations.
Similar results are proved for multi-type Markov branching processes 
by \citet[Section VIII.3]{AsmussenH}. 
Their type space may
be very general, so this setting includes also the single-type
non-Markov case studied
here (also in the non-lattice case \cite[Section VIII.12]{AsmussenH}), 
since a Crump--Mode--Jagers branching process may be seen as a Markov
process where the type of an individual is its entire life history until
present. However, this will in general be a large type space, and the
assumptions of \cite{AsmussenH} will in general not be satisfied;
in particular, their ``condition (M)'' \cite[p.~156]{AsmussenH}
is typically not satisfied, by the same argument as in
\cite[p.~173]{AsmussenH} for a related situation.
Hence, we can not obtain our results directly from the 
closely related results in \cite{AsmussenH}, although
there is an overlap in some special cases (for example the
Galton--Watson case in \refE{EGW}).

Another related situation is given by multi-colour \Polya{} urn processes,
see \eg{} \cite{SJ154} 
(which uses methods and results from branching process theory).
The same trichotomy appears there too, with a criterion formulated in terms
of eigenvalues of a matrix that can be seen as the (expected) ``offspring
matrix'' in that setting.

It would be interesting to find more general theorems that would include  
these different but obviously related results together.

\begin{rem}
  Our setup includes the Galton--Watson case, where all births occur when
  the mother has age 1 (\refE{EGW}), 
but this case is much simpler than the general case
  and can be treated by simpler methods; 
see
\citet[Section 2.10]{Jagers}, where results closely related to the ones below
are given.
\end{rem}

\begin{rem}
It would be very interesting to extend the results to the 
perhaps more interesting non-lattice case;
we expect similar results
(under suitable assumptions),
but this case seems to present new technical challenges, and we leave this
as an open problem.
\end{rem}

\section{Assumptions and main result}\label{Smain}

Let $\mu:=\E\Xi$ be  the intensity measure of the offspring process;
thus $\mu:=\sumk \mu_k\gd_k$, where 
$\mu_k$ is the expected number of children 
that an individual bears at age $k$ (and $\gd_k$ is the Dirac delta, i.e.,  a
point mass at $k$). 
Let $N_k:=\Xi\set k$ be the number of children born to an individual at age $k$.
Thus $N=\sumki N_k$ and $\mu_k=\E N_k$. 

We make the following standing assumptions, valid throughout the paper.
The first assumption (supercriticality) is essential; otherwise there is no
asymptotic behaviour to analyse.
The assumptions \ref{Amu0}--\ref{Adeath}
are simplifying and convenient but presumably not essential.
(For \ref{Adeath}, this is shown in \refE{Edeath}.)
\begin{Aenumerate}

\item \label{Afirst}
 The process is supercritical, \ie,
$\mu([0,\infty])=\sumk\mu_k = \E N>1$.

\item\label{Amu0}  No children are born instantaneously, \ie, $\mu_0=0$.

\item \label{Age1} $N\ge1$ a.s.
Thus the process \as{} survives.

\item \label{Adeath} There are no deaths.

\end{Aenumerate}

Define, for all complex $z$ such that either $z\ge0$ or 
the sums or expectations below converge
absolutely,  
\begin{equation}\label{mmu}
  \mmu(z):=\sumk \mu_k z^k 
=\sumk \E [N_k] z^k 
=\E \sumiN z^{\xi_i}
\end{equation}
and
the complex-valued random variable
\begin{equation}\label{mXi}
  \mXi(z):=
\intoo z^x\dd\Xi(x)
=\sumiN z^{\xi_i}
=\sumk N_k z^k.
\end{equation}
Thus $\mmu(z)=\E\mXi(z)$.

We make two other standing assumptions:
\begin{Aenumerateq}

\item \label{Amalthus}
$\mmu\bigpar{m\qw}=1$ for some $m>1$. 
\end{Aenumerateq}

Thus $\ga:=\log m$ satisfies 
$\sumki \mu_k e^{-k\ga}=\mmu(e^{-\ga})=1$, so $\ga$ is the
Malthusian parameter, and the population grows roughly with a factor
$e^{\ga}=m$ for each generation (see \eg{} \eqref{EZ} and \eqref{Zoo} below).

\begin{Aenumerateq}
\item \label{Ar}
$\E [\mXi(r)^2]<\infty$ for some $r>m\qqw$.
\end{Aenumerateq}

We fix in the sequel  some $r>m\qqw$ satisfying \ref{Ar}.
We assume for convenience $r\le1$.
Note that \ref{Ar} implies 
\begin{equation}\label{mmur}
\mmu(r)=\E\mXi(r)<\infty.  
\end{equation}
Hence $\mmu(z)$ and $\mXi(z)$ are defined, and analytic, 
at least for $|z|\le r$.
Since $\mmu(z)$ is a strictly increasing function on $\ooo$,
$m\qw$ in \ref{Amalthus} is the unique positive root of $\mmu(z)=1$.
However, $\mmu(z)=1$ may have other complex roots; we shall see that the
asymptotic behaviour of the fluctuations depends crucially on the position
of these roots. 
We define,
with $D_r:=\set{|z|<r}$,
\begin{align}\label{gGG}
  \gG&:=\set{z\in D_r:\mmu(z)=1},
\qquad \gGG:=\gG\setminus\set{m\qw},
\\
\label{ggx}
  \ggx&:=\inf \set{|z|:z\in\gGG},
\\
\gGGG&:=\set{z\in\gGG:|z|=\ggx}, \label{gGGG}
\end{align}
with $\ggx=\infty$ if $\gGG=\emptyset$.
(These sets may depend on the choice of $r$, but for our purposes this does
not matter. Recall that we assume $r>m\qqw$.) 
Since $\mmu(z)$ is analytic, $\gG$ is discrete and thus,
if $\ggx<\infty$, then $\gGGG$ is a finite non-empty set which we write as
$\set{\gam_1,\dots,\gam_q}$.


Let $Z_n$ be the total number of individuals at time $n$. 
(Which by \ref{Amu0} equals the number of individuals born up to time $n$.)
We define $Z_n$
for all integers $n$ by letting $Z_n:=0$ for $n<0$.
By assumption, $Z_0=1$.
It is well-known that the  number of individuals $ Z_n$
grows asymptotically
like $m^n$ as \ntoo.
For example, see \eg{} \cite[Theorem (6.3.3)]{Jagers} (and remember
that we here consider the lattice case),
\begin{equation}
  \label{EZ}
\E Z_n \sim \cc m^n \ccdef\ccEZ,
\qquad \text{ as \ntoo},
\end{equation}
with some $\ccEZ>0$.
Moreover, if $\E[\mXi(m\qw)\log\mXi(m\qw)]<\infty$, and
in particular if $\E [\mXi(m\qw)^2]<\infty$, which follows from our assumption 
\ref{Ar}, then
\begin{equation}\label{Zoo}
  Z_n/m^n\asto\cZ,
\qquad \text{ as \ntoo},
\end{equation}
for some random variable $\cZ>0$, see
\eg{} 
\citet{Nerman}. 
In particular, it follows 
that
for any fixed $k\ge1$ 
\begin{equation}\label{znkz}
  Z_{n-k}/Z_n\asto m^{-k}.
\end{equation}

 The number of individuals of age $\ge k$ at time $n$ is $Z_{n-k}$.
For large $n$, we expect this to be roughly $m^{-k} Z_n$, see \eqref{znkz},
and to study the fluctuations, we  define
\begin{equation}\label{ynk}
  X_{n,k}:=Z_{n-k}-m^{-k} Z_n,
\qquad k=0,1,\dots
\end{equation}
Note that $X_{n,0}=0$.

We state our main results
as three separate theorems, 
treating  the cases $\ggx> m\qqw$,
$\ggx=m\qqw$ and $\ggx<m\qqw$ separately.
In particular, note that Theorems \ref{T1}--\ref{T2}
yield asymptotic normality of $X_{n,k}$ when $\ggx\ge m\qqw$.
Proofs are given in later sections.
The results are extended to random characteristics in
\refS{Schar}. 

By the assumption \ref{Ar} and
\eqref{mXi}, $\E N_k^2<\infty$ for every $k\ge1$.
Define, for $j,k\ge1$, 
\begin{equation}\label{gsjk}
  \gs_{jk}:=\Cov(N_j,N_k)
\end{equation}
and, at least for $|z|< r$,
\begin{equation}\label{tc}
\gS(z):=  \sum_{i,j}\gs_{ij}z^i\bz^j
=\Cov\Bigpar{\sum_i N_iz^i,\sum_jN_j\bz^j}
=\E\bigabs{ \mXi(z)-\mmu(z)}^2.
\end{equation}
Let, for $R>0$, $\llr$ be the Hilbert space of infinite vectors
\begin{equation}\label{llr}
  \ell^2_R:=
\Bigset{(a_k)\koo:\norm{(a_k)\xoo}_{\llr}^2:=\sumk R^{2k}|a_k|^2<\infty}.
\end{equation}
(We often simplify the notation and denote a vector in $  \ell^2_R$ by
$(a_k)_k$.)

We begin with the case $\ggx>m\qqw$, which by \eqref{gGG}--\eqref{ggx}
is equivalent to:
\begin{Bnoenumerate}
\item \label{Aroots}
$\mmu(z)\neq 1$ for all complex $|z|\le m\qqw$ except possibly $z= m\qw$.
\end{Bnoenumerate}

\begin{thm}
  \label{T1}
  Assume \refAAA, i.e., $\ggx>m\qqw$. Then, as \ntoo,
\begin{equation}\label{t1a}
  X_{n,k}/\sqrt{Z_n} \dto \zeta_k,
\end{equation}
jointly for all $k\ge0$, 
for some jointly normal random variables $\zeta_k$ with mean $\E\zeta_k=0$ and
covariance matrix given by, for any finite sequence $a_0,\dots,a_K$ of real
numbers,
\begin{equation}\label{t1b}
\Var\Bigpar{\sum_k a_k\zeta_k}
=
\frac{m-1}m\oint_{|z|=m\qqw}\frac{\lrabs{\sum_k a_k z^k -\sum_k a_k m^{-k}}^2}
{|1-z|^2\,|1-\mmu(z)|^2}	
\gS(z)\frac{|\dz|}{2\pi m\qqw}.
\end{equation}

The convergence \eqref{t1a} holds also in the stronger sense that
$ ( Z_n\qqw X_{n,k})_k \dto (\zeta_k)_k$
in the Hilbert space $\llr$,
for any $R<m\qq$.
The limit variables $\zeta_k$ are non-degenerate unless
$\Xi$ is deterministic, i.e., $N_k=\mu_k$ \as{} for each $k\ge0$.
\end{thm}

Recall that joint convergence of an infinite number of variables means joint
convergence of any finite set. (This is convergence in the
product space $\bbR^\infty$, see \cite{Billingsley}.)
Note that trivially $\zeta_0=0$ (included for completeness).

The variance formula \eqref{t1b} can be interpreted as a stochastic
calculus, where the limit variables are seen as stochastic integrals (in a
general sense) of certain functions on the circle $|z|=m\qqw$; these
functions thus represent the random variables $\zeta_k$, and therefore
asymptotically $X_{n,k}$; moreover, they can be used for convenient
calculations. 
See \refS{Sstoch} for details.  

We give two proofs of \refT{T1}. The first, in Sections
\ref{Snormal}--\ref{SpfT1-A}, is based on the elementary central limit
theorem for sums of independent variables, together with some
approximations. 
This proof is extended to random characteristics in \refS{Schar}. 
The second proof is given in Sections \ref{Smart}--\ref{SpfT1-B}; it is
based on a martingale central limit theorem.
This proof easily adapts to give a proof of \refT{T2} below in \refS{SpfT2}.

We consider next the cases $\ggx\le m\qq$.
Then $\gGGG=\set{\gam_1,\dots,\gam_q}$
is a non-empty finite set.
For simplicity, we assume the
condition 
\begin{equation}
  \label{AT3}
\mmu'(\gam)\neq 0,
\qquad \gam\in\gGGG,
\end{equation}
\ie, that the points in $\gGGG$ are  simple
roots of $\mmu(z)=1$; the modifications in the case with a multiple root are
left to the reader. (See \refR{RAT3}, and note the related results for
\Polya{} urns in
\cite[Theorems 3.23--3.24]{SJ154} and \cite[Theorems 3.5--3.6]{Pouyanne}.)

\begin{thm}\label{T2}
  Assume \refAA{} and $\ggx=m\qqw$.
Suppose further that \eqref{AT3} holds.
Then, as \ntoo,
\begin{equation}\label{t2a}
  X_{n,k}/\sqrt{nZ_n} \dto \zeta_k,
\end{equation}
jointly for all $k\ge0$, 
for some jointly normal random variables $\zeta_k$ with mean $\E\zeta_k=0$ and
covariance matrix given by, for any finite sequence $a_0,\dots,a_K$ of real
numbers,
\begin{equation}\label{t2b}
  \Var\Bigpar{\sum_k a_k\zeta_k}
=
(m-1)\sumpq\frac{\lrabs{\sum_k a_k \gam_p^k -\sum_k a_k m^{-k}}^2}
{|1-\gam_p|^2\,|\mmu'(\gam_p)|^2}	
\gS(\gam_p).
\end{equation}
Moreover, the convergence \eqref{t1a} holds also
in the Hilbert space $\llr$,
for any $R<m\qq$.

The limit variables $\zeta_k$ are non-degenerate unless 
$\mXi(\gam_p)$ is deterministic for each $\gam_p\in\gGGG$.
\end{thm}

\begin{thm}\label{T3}
Assume \refAA{} and $\ggx<m\qqw$.
Suppose further that \eqref{AT3} holds.
Then
there exist complex random variables $U_1,\dots,U_q$ 
and linearly independent vectors 
$\vecu_i:=\bigpar{\gam_i^{k}-m^{-k}}_k$, 
$i=1,\dots,q$,
such that
\begin{equation}\label{t3}
  \ggx^n \vecx_n - \sumiq \bigpar{\bgam_i/|\gam_i|}^n U_i \vecu_i
\to 0
\end{equation}
\as{} and in $L^2(\llr)$, 
for any $R<m\qq$.
Furthermore,
$\E U_i=0$, and 
$U_i$ is non-degenerate unless $\mXi\bigpar{\gam_i}$ is degenerate. 
\end{thm}

Theorems \ref{T1}--\ref{T3} exhibit several differences between the cases 
$\ggx<m\qqw$, $\ggx=m\qqw$ and $\ggx>m\qqw$; 
\cf{} the similar results for \Polya{} urns
in \eg{} 
\cite[Theorems 3.22--3.24]{SJ154}.
\begin{itemize}
\item 
The fluctuations $X_{n,k}$, for a fixed $k$, 
are asymptotically normal when $\ggx\ge m\qqw$, but (presumably) not when
$\ggx<m\qqw$. 
\item 
The fluctuations 
are typically of order
$Z_n\qq\asymp m^{n/2}$ when $\ggx>m\qqw$, 
slightly larger (by a power of $n$) when $\ggx=m\qqw$, 
and of the much larger order
$\ggx^{-n}$ when $\ggx<m\qqw$.
\item 
When $\ggx<m\qqw$, the fluctuations exhibit oscillations that are periodic or
almost periodic (see \cite{Bohr})
in $\log n$.
(Note that $\gam_i/|\gam_i|\neq1$ in \eqref{t3}, since $m\qw$ is the only
positive root in $\gG$.)
\item 
When $\ggx<m\qqw$, there is the \as{} approximation result  \eqref{t3}, 
implying both long-range dependence as \ntoo, and that the asymptotic behaviour
essentially is determined by what happens in the first few generations.
In contrast,
the limits in \eqref{t1a} and \eqref{t2a} are mixing
(see  the proofs), \ie, the results holds also conditioned on the life
histories of the first $M$ individuals for any fixed $M$, and thus also
conditioned on $Z_1,\dots,Z_K$ for any fixed $K$;
hence, when $\ggx\ge m\qqw$, the initial behaviour is eventually forgotten.
Moreover
for $\ggx> m\qqw$, 
there is only a short-range dependence, 
see \refE{ESI1},
while the case $\ggx=m\qqw$ shows an intermediate
``medium-range'' dependence, see \refSS{SS2}.

\item 
When $\ggx>m\qqw$, the limit random variables $\zeta_k$ in \eqref{t1a} are
linearly independent, as a consequence of \eqref{t1b}. 
When $\ggx\le m\qqw$, the limits in \eqref{t2a}, or
the components of the sum in  \eqref{t3}, span a (typically) 
$q$-dimensional space of
random variables, and any $q+1$ of them are linearly
dependent; see also \refS{Sstoch}.
\end{itemize}

\begin{rem}\label{Rk<0}
  We consider above $X_{n,k}$ for $k\ge0$, \ie, the age distribution of the
  population at time $n$. 
We can define $X_{n,k}$ by \eqref{ynk} also for $k<0$; this means looking
into the future and can be interpreted as predicting the future population.
As shown in \refS{Sstoch}, \eqref{t1a}--\eqref{t1b} and
\eqref{t2a}--\eqref{t2b} extend to all $k\in\bbZ$ (still jointly), and,
similarly, taking the $k$th component in \eqref{t3} yields a result that
extends to all $k\in\bbZ$.

This enables us, for example, to obtain (by  standard linear algebra)
the best linear predictor of $Z_{n+1}$ based on the observed
$Z_{n},\dots,Z_{n-K}$ for any fixed $K$.
\end{rem}

\begin{ex}[Galton--Watson]\label{EGW}
  The simplest example is a Galton--Watson process, where all children are
  born in a single litter
at age 1 of the parent, so $N_k=0$ for $k\ge2$. 
(But all individuals live forever in our setting. In the traditional
setting, only the newborns are counted, i.e., $Z_n-Z_{n-1}$; the results are
easily transferred to this version.)
Then $N=N_1$,
$m=\mu_1$ and
  $\mmu(z)=mz$. 
Hence $\gG=\set{m\qw}$, $\gGG=\emptyset$, and $\ggx=\infty>m\qqw$.
We assume $\E N^2<\infty$; then \ref{Ar} holds for any $r$;
we also assume $N\ge1$ \as{} and $\PP(N>1)>0$;
then \refAAA{} hold.

Thus \refT{T1} applies.
We obtain, for example, with $\gss:=\Var(N)=\gs_{11}$,
\begin{equation}
  \begin{split}
\Var\bigpar{\zeta_1}
&=
\frac{m-1}m\oint_{|z|=m\qqw}\frac{\lrabs{z - m^{-1}}^2}
{|1-z|^2\,|1-mz|^2}	
\gss |z|^2 \frac{|\dz|}{2\pi m\qqw}
\\&
=
\gss\frac{m-1}{m^4}\oint_{|z|=m\qqw}\frac{1}
{|1-z|^2}	
 \frac{|\dz|}{2\pi m\qqw}
\\&
=\gss m^{-3}.
 \end{split}
\end{equation}
This can be shown directly in a much simpler way;
see  \cite[Theorem (2.10.1)]{Jagers}, which is
essentially equivalent to our \refT{T1} in the Galton--Watson case
(but without our assumption \ref{Age1}).
\end{ex}

\begin{ex}\label{E2}
Suppose that  all children are born when the mother has age one or two,
 \ie, $N_k=0$ for $k>2$.
Then $\mmu(z)=\mu_1z+\mu_2z^2$, where by assumption $\mu_1+\mu_2>1$ and
$\mu_1>0$.
\ref{Amalthus} yields $m^2=\mu_1m+\mu_2$, and thus
\begin{equation}\label{e2m}
  m=\frac{\mu_1+\sqrt{\mu_1^2+4\mu_2}}{2}.
\end{equation}
The equation $\mmu(z)=1$ has one other root, viz.\ $\gam_1$ with
\begin{equation}\label{e2gam}
  \gam_1\qw=-\frac{\sqrt{\mu_1^2+4\mu_2}-\mu_1}{2}.
\end{equation}
The condition \ref{Aroots} is thus equivalent to 
$|\gam_1|>m\qqw$, or
$\gam_1\qww<m$, which after
some elementary algebra is equivalent to, for example,
\begin{equation}\label{g2}
\mu_1^3+3\mu_1\mu_2+\mu_2-\mu_2^2>0.
\end{equation}
Thus, \refT{T1} applies when \eqref{g2} holds, \refT{T2} when there is
equality in \eqref{g2}, and \refT{T3} when the \lhs{} of \eqref{g2} is
negative. (In this example, \eqref{AT3} is trivial.)

For a simple numerical example with $\ggx=m\qqw$, take $\mu_1=2$ and $\mu_2=8$.
Then \eqref{e2m}--\eqref{e2gam} yield $m=4$ and $\gam_1=-\frac12$.
We obtain by \eqref{t2b}, for example,
\begin{equation}\label{e2ex}
X_{n,1}/\sqrt{nZ_n}\dto \zeta_1\sim N\Bigpar{0,\frac{1}{768}\Var(N_2-2N_1)}.   
\end{equation}

Suppose now instead that \eqref{g2} holds, so \refT{T1} applies. 
Let $\gl:=\gam_1\qw$ be given by \eqref{e2gam}. Then
$1-\mmu(z)=(1-mz)(1-\gl z)$, and thus \eqref{t1b} yields, for example,
\begin{equation}\label{e2t1b}
  \begin{split}
\Var\bigpar{\zeta_1}
&=
\frac{m-1}m\oint_{|z|=m\qqw}\frac{\lrabs{ z -m^{-1}}^2}
{|1-z|^2\,|1-\mmu(z)|^2}	
\gS(z)\frac{|\dz|}{2\pi m\qqw}
\\&=    
\frac{m-1}{m^3}\oint_{|z|=m\qqw}
\frac{\gs_{11}|z|^2+\gs_{12}(z+\bz)|z|^2+\gs_{22}|z|^4}
{|1-z|^2\,|1-\gl z|^2}	
\frac{|\dz|}{2\pi m\qqw}.
  \end{split}
\end{equation}
This integral can be evaluated by expanding $(1-z)\qw(1-\gl z)\qw$ in a
Taylor series; this yields after some calculations
\begin{equation}
  \Var\bigpar{\zeta_1}
=
\frac{(m+\gl)(\gs_{11}+\gs_{22}/m) + 2 (1+\gl)\gs_{12}}
{m^2(m-\gl)(m-\gl^2)}.
\end{equation}
\end{ex}

\begin{rem}
  The limit in \eqref{t1a} is by \refT{T1} 
degenerate only when the entire process is, and thus each $X_{n,k}$ is
degenerate.
In contrast, the limit in \eqref{t2a} or the approximation in
\eqref{t3} may be degenerate even in
other (special) situations. For example, let $N_1$ be non-degenerate with
$\E N_1=2$, 
let $N_2:=2N_1+4$, and let $N_k:=0$ for $k>2$. 
Then $\mu_1=2$ and $\mu_2=8$, and \refE{E2} shows that $\ggx=\frac12=m\qqw$;
furthermore, 
\eqref{e2ex}  applies and yields $X_{n,k}/\sqrt{nZ_n}\dto 0$.

We conjecture that in this case 
(and similar ones with $\zeta_k=0$ in \refT{T2}), 
$X_{n,k}/\sqrt{Z_n}$ has a non-trivial normal limit in distribution;
we leave this as an open problem.
Simliarly, we conjecture that when each $\mXi(\gam_i)$ is degenerate in
\refT{T3}, the distribution of $X_{n,k}$ is asymptotically determined by the
next smallest roots in $\gGG$.
\end{rem}

\subsection{More notation}

For a random variable $X$ in a Banach space $\cB$,  
we define $\norm{X}_{L^2(\cB)}:=(\E \norm{X}_{\cB}^2)\qq$,
when $\cB=\bbR$ or $\bbC$ abbreviated to
$\norm{X}_2$.

For infinite vectors $\vec x=(x_j)\joo$ and $\vec y=(y_j)\joo$, let
$\innprod{\vec x,\vec y}:=\sumj x_j y_j$, assuming that the sum converges
absolutely. 

$C$ denotes different constants that may depend on the distribution
of the branching process (\ie, on the distribution of $N$ and $(\xi_i)$),
but not on $n$ and similar parameters; the constant may change from one
occurrence to the next.

$\Oas(1)$ means a quantity that is bounded by a random constant that does
not depend on $n$.

All unspecified limits are as \ntoo.

\section{Preliminaries}\label{Sprel}

 Let 
 \begin{equation}
   \label{bn}
B_n:=Z_n-Z_{n-1}
 \end{equation}
be the number of individuals born at
time $n$ (with $B_0=Z_0$). Thus,
\begin{equation}\label{zzk}
  Z_{n}=Z_{n-1}+B_n=\sum_{i=0}^n B_i,
\qquad n\ge0.
\end{equation}

Let $B_{n,k}$ be the number of individuals born at time $n+k$ by parents
that are themselves born at time $n$, and thus are of age $k$.
Thus,   recalling \ref{Amu0}, 
\begin{equation}\label{xnk}
  B_n=\sumkin B_{n-k,k},
\qquad n\ge1.
\end{equation}
Let $\FF_n$ be the $\gs$-field generated by the life histories of all
individuals 
born up to time $n$, with $\FF_{n}$ trivial for $n<0$.
Then $B_{n,k}$ is $\FF_n$-measurable, and $B_n$ is
$\FF_{n-1}$-measurable by \eqref{xnk}.
Furthermore,
\begin{equation}\label{wnk0}
  \E\bigpar{B_{n,k}\mid \FF_{n-1}} = \mu_k B_{n},
\qquad n\ge0.
\end{equation}
For  $k\ge1$, let 
\begin{equation}\label{wnk}
  W_{n,k}:=
B_{n,k}- \E\bigpar{B_{n,k}\mid \FF_{n-1}}
=B_{n,k}-\mu_k B_{n}.
\end{equation}
(Thus $W_{n,k}=0$ if $n<0$.)
Then $W_{n,k}$ is $\FF_n$-measurable with
\begin{equation}\label{mgw}
\E\bigpar{W_{n,k}\mid\FF_{n-1}}=0.  
\end{equation}
Let further 
\begin{equation}\label{wn0}
  W_n:= B_n-\sumkin \mu_kB_{n-k}
= B_n-\sumki \mu_kB_{n-k}.
\end{equation}
Thus $W_0=B_0=Z_0$, and for $n\ge1$,
by
\eqref{wn0}, \eqref{xnk} and \eqref{wnk},
\begin{equation}\label{wn1}
  W_n=\sumkin W_{n-k,k}.
\end{equation}

\begin{lem}
  \label{LA}
Assume \refAA{}.
Then, for all $n\ge1$ and $k\ge1$,
$\E [W_{n,k}^2]\le C r^{-2k} m^n$
and
$\E [W_n^2]\le C m^n$. 
\end{lem}
\begin{proof}
Recall that $N_k$ is the number of children born at age $k$ of an
individual, and that $\E N_k=\mu_k$.
Furthermore, by \eqref{mXi}, $\mXi(r) \ge N_k r^k$ and thus
\begin{equation}\label{evk2}
\Var N_k\le
  \E N_k^2 \le r^{-2k} \E [\mXi(r)^2] = \CC r^{-2k}.
\CCdef\CCevk
\end{equation}

Let $n\ge0$ and $k\ge1$. Given $\FF_{n-1}$, $B_{n,k}$ is the sum of $B_{n}$
independent copies of $N_k$, and thus, 
see \eqref{wnk}, \eqref{wnk0} and \eqref{evk2},
\begin{equation}
  \E \bigpar{W_{n,k}^2\mid\FF_{n-1}}
=B_{n} \Var(N_k) 
\le \CCx r^{-2k} B_{n}.
\end{equation}
Taking the expectation and using \eqref{EZ} we find
\begin{equation}\label{eleonora}
  \E [W_{n,k}^2]
\le \CCx r^{-2k} \E B_{n}
\le \CCx r^{-2k} \E Z_{n}
\le \CC r^{-2k} m^{n},
\end{equation}
as asserted. 
Consequently $\norm{W_{n,k}}_2\le \CC r^{-k} m^{n/2}$
and, by \eqref{wn1} and Minkowski's inequality,
using $rm\qq>1$,
\begin{equation}
  \norm{W_n}_2 \le \sumkin \norm{W_{n-k,k}}_2
\le \CCx m^{n/2}\sumki (rm\qq)^{-k}
\le \CC m^{n/2}.
\end{equation}
\end{proof}

For $n\ge0$ and $k\ge1$,
by \eqref{ynk},
\begin{equation}\label{ta}
  \begin{split}
X_{n+1,k}
&=Z_{n+1-k}-m^{-k}Z_{n+1}
=X_{n,k-1}+m^{1-k}Z_n-m^{-k}Z_{n+1}
\\&
=X_{n,k-1}+m^{-k}(mZ_n-Z_{n+1}).	
  \end{split}
\end{equation}
Furthermore, 
by \eqref{bn} and \eqref{ynk}, we have, for $k\ge0$,
\begin{equation}\label{xn-k}
  B_{n-k}=Z_{n-k}-Z_{n-k-1}
=X_{n,k}-X_{n,k+1}+(m-1)m^{-k-1}Z_n.
\end{equation}
By
\eqref{zzk}, \eqref{wn0} and \eqref{xn-k},
recalling that $X_{n,0}=0$ by \eqref{ynk} and $\mmu(m\qw)=1$ by \ref{Amalthus},
for $n\ge0$,
\begin{align}\label{tb}
mZ_n-Z_{n+1}
&= (m-1)Z_n-B_{n+1}
= (m-1)Z_n-\sumki\mu_k B_{n+1-k}-W_{n+1}
\notag\\&
= (m-1)Z_n-\sumki\mu_k \bigpar{X_{n,k-1}-X_{n,k}+(m-1)m^{-k}Z_n}
-W_{n+1}
\notag\\&
= (m-1)Z_n-\sumki\mu_k \bigpar{X_{n,k-1}-X_{n,k}}
-(m-1)\mmu(m\qw)Z_n
-W_{n+1}
\notag\\&
= \sumki\mu_k \bigpar{X_{n,k}-X_{n,k-1}}
-W_{n+1}
.
\end{align}
Consequently, \eqref{ta} yields, for $n\ge0$ and $k\ge1$,
\begin{equation}
  \label{tab}
X_{n+1,k}
=X_{n,k-1}+m^{-k}\biggpar{\sumji\mu_j \bigpar{X_{n,j}-X_{n,j-1}}-W_{n+1}}.
\end{equation}

Introduce the vector notation $\vecx_{n}:=(X_{n,k})_{k=0}^\infty$ and
\begin{equation}\label{v}
\vec v:=(0,m\qw,m\qww,\dots)
=\bigpar{m^{-k}\ett{k>0}}\koo,  
\end{equation}
and for vectors $\vec y=(y_k)_0^\infty$ such that the sum converges, define
\begin{equation}
  \label{chi}
\chix\bigpar{(y_k)_0^\infty}:=
\sumki \mu_k(y_{k}-y_{k-1}).
\end{equation}
Let $S$ be the shift operator
$S\bigpar{(y_k)\xoo}:=(y_{k-1})\xoo$ with $y_{-1}:=0$,
and let $T$ be the linear operator
\begin{equation}\label{T}
  T\xpar{\vec y}:=S(\vec y)+\chix(\vec y)\vec v.
\end{equation}
Then \eqref{tab} can be written,
again recalling that $X_{n,0}=0$,
\begin{equation}\label{yn+}
  \vecx_{n+1}
  =S(\vecx_n) + \bigpar{\chix(\vecx_n) -W_{n+1}}\vec v
=T(\vecx_n)-W_{n+1}\vec v.
\end{equation}
This recursion leads to 
the following formula.
\begin{lem}
  \label{LY}
For every $n\ge0$,
\begin{equation}\label{ly}
  \vecx_n=-\sumkn W_{n-k}T^k(\vec v).
\end{equation}
\end{lem}
\begin{proof}
For the
initial value $\vecx_0$, we have by \eqref{ynk} $X_{0,k}=-m^{-k}Z_0$ for
$k\ge1$, and thus by \eqref{v} $\vecx_0=-Z_0\vec v=-W_0\vec v$, recalling
that $W_0=B_0=Z_0$. This verifies \eqref{ly} for $n=0$. The general case
follows by \eqref{yn+} and induction.  
\end{proof}

\begin{rem}\label{RLY}
  It follows from the proofs below, that 
the sum in \eqref{ly} is dominated by 
the first few terms in the case $\ggx>m\qqw$, 
and by the last few terms in the case $\ggx<m\qqw$, 
while all terms are of about the same size when $\ggx=m\qqw$.
This explains much of the different behaviours seen in \refS{Smain}.
\end{rem}

We now consider $T$ defined in \eqref{T}
as an operator on the complex Hilbert space 
$  \ell^2_R$ defined in \eqref{llr}
for a suitable $R>0$.
Recall that the \emph{spectrum} $\gs(T)$ of a linear operator in a complex
Hilbert
(or Banach) space is the set of complex numbers $\gl$ such that $\gl-T$ is
\emph{not} invertible; see \eg{} \cite[Section VII.3]{Dunford-Schwartz}.

\begin{lem}\label{LB}
  Suppose that $1\le R<m$ and that $\mmu(R\qw)<\infty$.
Then $\vec v\in \ell^2_R$, $\chix$ is a bounded linear functional on
 $\ell^2_R$ and $T$ is a bounded linear operator on $\ell^2_R$.
Furthermore, if\/ $\gl\in\bbC$ with $|\gl|>R$, then
$\gl\in\gs(T)$ if and only if\/ 
$\gl\qw\in\gGG$,
\ie, if and only if\/ 
$\gl\neq m$ and 
$\mmu(\gl\qw)= 1$.
\end{lem}

\begin{proof}
By \eqref{v} and \eqref{llr}, and because $R<m$,
\begin{equation}
  \norm{\vec v}_{\llr}^2 = \sumki R^{2k}m^{-2k}<\infty.
\end{equation}
Next, 
it is clear from \eqref{llr} that the shift operator $S$ is bounded on
$\llr$ (with norm $R$).
Furthermore, by \eqref{mmu} and assumption,
\begin{equation}
  \sumki R^{-2k}\mu_k^2 \le \mmu(R\qw)^2<\infty
\end{equation}
and it follows by the Cauchy--Schwarz inequality that
$\chix_1\bigpar{(a_k)\xoo}:=\sumki\mu_ka_k$ defines a bounded linear
functional $\chix_1$ on $\llr$. Since $\chix$ can be written $\chix=\chix_1-\chix_1
S$, $\chix$ too is bounded.
It now follows from \eqref{T} that $T$ is a bounded linear operator on $\llr$.
  
For the final statement we note that the mapping $(a_k)\xoo\mapsto\sumk a_k z^k$
is an isometry of $\llr$ onto the Hardy space $H^2_R$ consisting of all
analytic functions $f(z)$ in the disc \set{z:|z|<R} such that
\begin{equation}\label{hhr}
  \norm{f}_{\hhr}^2
:=\sup_{r<R} \frac{1}{2\pi}\int_0^{2\pi}|f\bigpar{re^{i\theta}}|^2\,d\theta
<\infty.
\end{equation}
(See \eg{} \cite{Duren}.) 
In particular, $\vec v$ corresponds to the function 
\begin{equation}\label{vz}
  v(z):=\sumki m^{-k}z^k = \frac{z/m}{1-z/m}=\frac{z}{m-z}.
\end{equation}
We use the same notations $\chix$, $S$ and $T$ for the corresponding linear
functional and operators on $\hhr$, and note that
the shift operator $S$ on $\llr$ corresponds to the multiplication operator 
$S f(z)= zf(z)$ on $\hhr$. The definition \eqref{T} thus translates to
\begin{equation}
  \label{TH}
Tf(z)=zf(z)+\chix(f)v(z). 
\end{equation}
Consequently, for any $h\in\hhr$, the equation
$(\gl-T)f=h$ is equivalent to
\begin{equation}\label{winston}
  (\gl-z)f(z)-\chix(f)v(z)=h(z).
\end{equation}
Any solution to \eqref{winston} has to be of the form
\begin{equation}\label{magnus}
  f(z)=c\frac{v(z)}{\gl-z}+\frac{h(z)}{\gl-z},
\end{equation}
where
\begin{equation}\label{anna}
  c=\chix(f)
=c\chix\Bigpar{\frac{v(z)}{\gl-z}}+\chix\Bigpar{\frac{h(z)}{\gl-z}}.
\end{equation}
Suppose $|\gl|>R$; then $1/(\gl-z)$ is a bounded analytic function on the
domain \set{|z|<R}, so it follows from \eqref{hhr} and $v,h\in\hhr$ that
$v(z)/(\gl-z)\in\hhr$ and $h(z)/(\gl-z)\in\hhr$.
If $\chix\bigpar{v(z)/(\gl-z)}\neq1$, then \eqref{anna} has a unique solution
$c$ for any $h\in\hhr$, and thus \eqref{winston} has a unique solution
$f\in\hhr$, given by \eqref{magnus}. 
In other words, then $\gl-T$ is invertible on $\hhr$ and $\gl\notin\gs(T)$.
(Continuity of $(\gl-T)\qw$ is automatic, 
by the closed graph theorem.)
Conversely, if
$\chix\bigpar{v(z)/(\gl-z)}=1$, then \eqref{winston} has either no solution
or
infinitely many solutions $f$ for any given $h\in\hhr$, and thus $\gl\in\gs(T)$.

We have shown that for $|\gl|>R$,
\begin{equation}\label{spectrum}
  \gl\in\gs(T) \iff \chix\Bigparfrac{v(z)}{\gl-z}=1
. 
\end{equation}

We analyse the condition in \eqref{spectrum} further.
If $|\gl|>R$ and $\gl\neq m$, then, by \eqref{vz},
\begin{equation}\label{vix}
  \begin{split}
  \frac{v(z)}{\gl-z}
&=
  \frac{z}{(\gl-z)(m-z)}
=\frac{1}{m-\gl}\Bigpar{\frac{\gl}{\gl-z}-\frac{m}{m-z}}.
  \end{split}
\end{equation}
Furthermore,
$\gl/(\gl-z)=\sumk \gl^{-k}z^k$ and thus by \eqref{chi} and \eqref{mmu},
\begin{equation}
\chix\Bigpar{\frac{\gl}{\gl-z}}
=\sumki \mu_k \gl^{-k}(1-\gl)
=(1-\gl)\mmu\bigpar{\gl\qw}.
\end{equation}
Hence, \eqref{vix} yields,
recalling $\mmu(m\qw)=1$ by \ref{Amalthus},
\begin{equation}\label{bix}
  \begin{split}
\chix\Bigpar{\frac{v(z)}{\gl-z}}	
&=\frac{1}{m-\gl}
\lrpar{\chix\Bigpar{\frac{\gl}{\gl-z}}-\chix\Bigpar{\frac{m}{m-z}}}
\\&
=\frac{1}{m-\gl}\bigpar{(1-\gl)\mmu(\gl\qw)-(1-m)\mmu(m\qw)}
\\&
=\frac{1}{m-\gl}\bigpar{(1-\gl)\mmu(\gl\qw)+m-1}.
  \end{split}
\end{equation}
Consequently, for $|\gl|>R$ with $\gl\neq m$,
by \eqref{spectrum} and \eqref{bix},
\begin{equation}
  \begin{split}
\gl\in\gs(T)
&\iff
\chix\Bigpar{\frac{v(z)}{\gl-z}}	=1
\\&
\iff	
(1-\gl)\mmu(\gl\qw)+m-1=m-\gl
\\&
\iff	
(1-\gl)\mmu(\gl\qw)=1-\gl
\\&
\iff
\mmu(\gl\qw)=1.
  \end{split}
\end{equation}
In the special case $\gl=m$, we find by continuity, letting $\gl\to m$ in
\eqref{bix}, 
\begin{equation}\label{mix}
  \begin{split}
\chix\Bigpar{\frac{v(z)}{m-z}}	
&=
\lim_{\gl\to m}\chix\Bigpar{\frac{v(z)}{\gl-z}}	
=-\frac{d}{d\gl}\bigpar{(1-\gl)\mmu(\gl\qw)}\big|_{\gl=m}
\\&
=\mmu(m\qw)-(m-1)m^{-2}\mmu'(m\qw)
<\mmu(m\qw)=1
  \end{split}
\end{equation}
since $\mmu'(x)>0$ for $x>0$. Hence $m\notin\gs(T)$.
\end{proof}

\begin{rem}\label{Rspectrum}
It is easily seen that $\gl\in\gs(T)$ for every $\gl$ with $|\gl|\le R$,
\eg{}  by taking $h=v$ in \eqref{winston}--\eqref{magnus} and noting that
  $v(z)/(\gl-z)\notin\hhr$. Thus we have a complete description of the
spectrum $\gs(T)$ on $\llr$.
\end{rem}

\begin{lem}\label{LC}
  Suppose that $1\le R<m$ and that $\mmu(R\qw)<\infty$.
Suppose furthermore that $\mmu(z)\neq 1$ for every complex $z\neq m\qw$ 
with $|z|<R\qw$. Then, for every $R_1>R$, there exists $C=C(R_1)$ such that 
\begin{equation}\label{lc}
  \norm{T^n}_{\llr} \le C R_1^n, 
\qquad n\ge0.
\end{equation}
\end{lem}

\begin{proof}
  By \refL{LB}, $T$ is a bounded linear operator on $\ell^2_R$ and 
if $\gl\in\gs(T)$ with $|\gl|>R$, then $\mmu(\gl\qw)=1$ and $\gl\qw\neq
m\qw$. By assumption, there is no such $\gl$, and thus
$\gs(T)\subseteq\set{\gl:|\gl|\le R}$. 
(Actually, equality holds by \refR{Rspectrum}.)
In other words, the
spectral radius
\begin{equation}
  \label{spectrhox}
\rhox(T):=\sup_{\gl\in\gs(T)}|\gl|\le R.
\end{equation}
By the spectral radius formula 
\cite[Lemma VII.3.4]{Dunford-Schwartz},
$\rhox(T)=\lim_\ntoo\norm{T^n}^{1/n}$
and thus \eqref{spectrhox} implies that, for any $R_1>R$,
$\norm{T^n}^{1/n}<R_1$ for large $n$, which 
yields \eqref{lc}.
\end{proof}

We shall use \refL{LC} when $\ggx>m\qqw$. In the case $\ggx\le m\qqw$, we
use instead the following lemma, based on a more careful spectral analysis
of $T$. Recall the definitions \eqref{gGG}--\eqref{gGGG}.

\begin{lem}\label{LD}
Assume that $R=r\qw\ge1$, where $\mmu(r)<\infty$.
Suppose furthermore that
$\gGGG=\set{\gam_1,\dots,\gam_q}\neq\emptyset$, and that \eqref{AT3} holds.
Let $\gl_i:=\gam_i\qw$.
Then there exist eigenvectors $\vecv_i$ with $T\vecv_i=\gl_i\vecv_i$
and linear projections $P_i$ with range $\cR(P_i)=\set{c\vecv_i:c\in\bbC}$ 
(i.e., the span of $\vecv_i$), 
$i=1,\dots q$, 
and furthermore a bounded operator $T_0$ in $\llr$ and a constant $\tR<\ggx\qw$
such that, for any $n\ge0$, 
\begin{equation}\label{toa}
  T^n = T_0^n + \sumiq \gl_i^n P_i
\end{equation}
and
\begin{equation}\label{ton}
  \bignorm{T_0^n}_{\llr}\le C \tR^n.
\end{equation}
Explicitly,
\begin{equation}\label{vecvi}
  \vecv_i
=P_i(\vecv)  
=\frac{1}{\gam_i(\gam_i-1)\mmu'(\gam_i)}\bigpar{\gam_i^{k}-m^{-k}}_k.
\end{equation}
\end{lem}

\begin{proof}
Since the points in $\gGG$ are isolated,
there is a number $\trx>\ggx$ such that
$|z|>\trx$ for any $z\in\gGG\setminus\gGGG$.
We may assume $\trx<r$. Let $\tR:=\trx\qw>R$.
By \refL{LB}, $\gl_i=\gam_i\qw\in\gs(T)$ with $|\gl_i|=\ggx\qw$, and
$|\gl|<\tR<\ggx\qw$ for any $\gl\in\gs(T)\setminus\set{\gl_1,\dots,\gl_q}$.

Since $\gl_1,\dots,\gl_q$ thus are isolated points in $\gs(T)$,
by standard functional calculus, see \eg{} 
\cite[Section VII.3]{Dunford-Schwartz}, 
there exist commuting projections (not necessarily orthogonal)
$P_0,\dots,P_q$ in $\llr$ such that $\sum_{i=0}^q P_i=1$, 
$T$ maps each subspace $E_i:=P_i(\llr)$ into itself, and if 
$\hT_i$ is 
the restriction of $T$ to $E_i$, 
then $\hT_i$ has spectrum $\gs(\hT_i)=\set{\gl_i}$ for $1\le i\le q$
  and $\gs(\hT_0)=\gs(T)\setminus\set{\gl_i}_1^q$.
In particular, the spectral radius
$\rhox(\hT_0)<\tR$, and thus, 
by the spectral radius formula
\cite[Lemma VII.3.4]{Dunford-Schwartz},
\begin{equation}\label{tono}
  \norm{\hT_0^n}\le C \tR^n,
\qquad n\ge0.
\end{equation}
Let $T_0:=TP_0$. Then $T_0^n=T^nP_0=\hT_0^nP_0$, and \eqref{ton} follows.

It remains to show
that the spaces $E_i=\cR(P_i)$ are one-dimensional, and spanned by the
vectors $\vecv_i$ in \eqref{vecvi}.

We use, as the proof of \refL{LB}, 
the isometry $(a_k)\xoo\mapsto\sumk a_k z^k$
of $\llr$ onto  $H^2_R$.
    
For each $\gl_i$, $\mmu(\gl_i\qw)=1$, and thus
$\chix\bigpar{v(z)/(\gl_i-z)}=1$ by \eqref{bix}, 
see also \eqref{spectrum}.
Hence, \eqref{winston}--\eqref{anna} show, by taking $h=0$, that
the kernel $\cN(\gl_i-T)$ is one-dimensional and spanned by $v(z)/(\gl_i-z)$.
Similarly, again by \eqref{winston}--\eqref{anna},
the range $\cR(\gl_i-T)$ 
is given by
\begin{equation}\label{RglT}
  \cR(\gl_i-T)=
\Bigset{h\in\llr: \chix\Bigpar{\frac{h(z)}{\gl_i-z}}=0}. 
\end{equation}
By differentiating \eqref{bix}, we find for $|\gl|>R$ with 
$\gl\qw\in\gGG$,
i.e., $\gl\neq m$ and $\mmu\bigpar{\gl\qw}=1$,
\begin{equation}\label{byx}
  \begin{split}
\chix\Bigpar{\frac{v(z)}{(\gl-z)^2}}	
&=-\ddl
\chix\Bigpar{\frac{v(z)}{\gl-z}}	
=\ddl\Bigpar{1-\chix\Bigpar{\frac{v(z)}{\gl-z}}}	
\\&
=\ddl\frac{(1-\gl)(1-\mmu(\gl\qw))}{m-\gl}
=\frac{(1-\gl)\mmu'(\gl\qw)}{(m-\gl)\gl^2}.
  \end{split}
\end{equation}
Thus, the assumption \eqref{AT3} implies that 
$\chix\bigpar{\xfrac{v(z)}{(\gl_i-z)^2}}	\neq0$, and thus 
$\xqfrac{v(z)}{\gl_i-z}\notin\cR(\gl_i-T)$ by \eqref{RglT}. 
Hence, $\cN(\gl_i-T)\cap\cR(\gl_i-T)=\set0$.
Consequently, for every $h\in\cR(\gl_i-T)$, \eqref{winston} has a unique
solution $f\in\cR(\gl_i-T)$, i.e., 
the restriction of $\gl_i-T$ to $\cR(\gl_i-T)$ is invertible.

It follows that the projection $P_i$ is the projection onto 
$\cN(\gl_i-T)=\set{cv(z)/(\gl_i-z)}$ that vanishes on $\cR(\gl_i-T)$, which 
by \eqref{RglT} is given by
\begin{equation}\label{pif}
  P_i(f(z))=\frac{\chix\bigpar{f(z)/(\gl_i-z)}}
{\chix\bigpar{\xfrac{v(z)}{(\gl_i-z)^2}}}
\cdot
\frac{v(z)}{\gl_i-z}.
\end{equation}
In particular, since $\chix\bigpar{v(z)/(\gl_i-z)}=1\neq0$,
$P_i(v)$ is a non-zero multiple of $v(z)/(\gl_i-z)$.
Let $\vecv_i:=P_i(\vec v)$. Thus $T \vecv_i=\gl_i\vecv_i$,
and, for $n\ge0$,
\begin{equation}
T^n =T^nP_0+\sumiq T^nP_i 
=T_0^n+\sumiq\gl_i^nP_i,  
\end{equation}
showing \eqref{toa}.

Finally, \eqref{pif} and \eqref{byx} yield
\begin{equation}
  v_i(z):=P_i\bigpar{v(z)}
=
\frac{(m-\gl_i)\gl_i^2}{(1-\gl_i)\mmu'(\gl_i\qw)}\cdot\frac{v(z)}{\gl_i-z},
\end{equation}
and \eqref{vecvi} follows because $\gl_i=\gam_i\qw$ and
by \eqref{vz}, for $|\gl|>R$,
\begin{equation}\label{mmix}
  (m-\gl)\frac{v(z)}{\gl-z}=\frac{\gl}{\gl-z}-\frac{m}{m-z}=
\sumk (\gl^{-k}-m^{-k})z^k  . 
\end{equation}
\end{proof}

\begin{rem}
  \label{RAT3}
It follows also that \eqref{AT3} implies that
the  points $\gl_i\in\gs(T)$ are simple poles of the resolvent
$(\gl-T)\qw$, and conversely. 
\refL{LD} can be extended without assuming \eqref{AT3}; the general result
is similar but more complicated, and is left to the reader.
Cf.\ \cite[Theorem VII.3.18]{Dunford-Schwartz}.
\end{rem}

We shall also use another similar calculation.
\begin{lem}  \label{Lres}
  Suppose that $1\le R<m$ and that $\mmu(R\qw)<\infty$.
If\/ $|\gl|>R$ and\/ $\mmu\bigpar{\gl\qw}\neq1$, then
\begin{equation}\label{resT}
(\gl-T)\qw(\vec v)
=\frac{1}{(1-\gl)(1-\mmu(\gl\qw))}\bigpar{\gl^{-\ellk}-m^{-\ellk}}_\ellk  .
\end{equation}
\end{lem}

\begin{proof}
  Taking $h=v$ in \eqref{winston}--\eqref{anna}, 
we find
\begin{equation}
(\gl-T)\qw v(z)=  f(z)=b\frac{v(z)}{\gl-z}
\end{equation}
for a constant $b$ such that $b=\chix(f)+1$.
This yields by \eqref{bix}
\begin{equation}
b-1=\chix(f)=\frac{b}{m-\gl}\bigpar{(1-\gl)\mmu(\gl\qw)+m-1}
\end{equation}
with the solution
\begin{equation}
  b=\frac{m-\gl}{(1-\gl)(1-\mmu(\gl\qw))}.
\end{equation}
Hence, using \eqref{mmix}, for $|z|<R$,
\begin{equation}\label{fix}
  \begin{split}
  f(z)
&=b\frac{v(z)}{\gl-z}
=\frac{1}{(1-\gl)(1-\mmu(\gl\qw))}\sumk (\gl^{-\ellk}-m^{-\ellk})z^\ellk  . 
  \end{split}
\end{equation}
\end{proof}

\section{A first normal convergence result}\label{Snormal}

Let
$\vec\eta:=(\eta_0,\eta_1,\eta_2,\dots)$, where $(\eta_k)\xoo$ are jointly
normal 
random variables with means $\E\eta_k=0$ and covariances
\begin{equation}
  \label{coveta}
\Cov(\eta_j,\eta_k)=\gs_{jk}=\Cov(N_j,N_k),
\end{equation}
see \eqref{gsjk}.
Note that $\eta_0=0$ since $N_0=0$.

\begin{lem}\label{Leta}
Assume \refAA, and
  let $\veceta\kk=(\eta\kk_j)_{j=0}^\infty$, 
$k=1,2,\dots$, be independent copies of the random
  vector $\eta$. Then, as \ntoo,
  \begin{equation}\label{leta}
	Z_n\qqw W_{n-k,j}\dto (1-1/m)\qq m^{-k/2}\eta_j\kk,
  \end{equation}
jointly for all $(j,k)$ with $j\ge0$ and $k\ge0$.
\end{lem}

\begin{proof}
  Consider first a fixed $k\ge0$.
Given $B_{n-k}$, the vector $\vec B_{n-k}:=(B_{n-k,j})\joo$ is the sum of
$B_{n-k}$ independent copies of the random vector $\vecN$, and by
\eqref{wnk},
the vector $\vec W_{n-k}:=(W_{n-k,j})\joo$ is the sum of
$B_{n-k}$ independent copies of the centered random vector $\vecN-\E\vec N$.
By \eqref{bn} and \eqref{znkz},
\begin{equation}\label{pyret}
  \frac{B_n}{Z_n}=1-\frac{Z_{n-1}}{Z_n}\asto 1-m\qw>0.
\end{equation}
In particular, $B_n\to\infty$ \as, and thus
$B_{n-k}\to\infty$. Consequently, by the central limit theorem for
\iid{} finite-dimensional vector-valued random variables, and the definition of 
$\eta_j$,
\begin{equation}\label{erika}
  B_{n-k}\qqw W_{n-k,j}\dto \eta_j \eqd \eta_j\kk,
\end{equation}
jointly for any finite set of $j\ge0$.

Moreover, by \eqref{pyret} and \eqref{znkz}, 
\begin{equation}
  \label{per}
B_{n-k}/Z_n\asto (1-1/m)m^{-k},
\end{equation}
and thus \eqref{leta} for a fixed $k$ 
follows from \eqref{pyret} and \eqref{erika}.

To extend this to several $k$, the problem is that $W_{n-k,j}$ for different
$k$ are, in general, dependent. (For example, conditioned on $Z_{n-1}$
and $B_{n-1}$, $W_{n-1,1}$ determines $B_{n-1,1}$ which contributes to
$B_n$, and thus influences $W_{n,j}$.) We therefore approximate $W_{n-k,j}$
as follows.

We may assume that for each $k$, we have an infinite sequence 
$(\vecN\kki)_{i\ge1}$ of independent copies of $\vecN$, 
such that $\vec W_{n-k}$ is the sum $\sum_{i=1}^{B_{n-k}}\vecN\kki$ of the
first $B_{n-k}$ vectors; furthermore, these sequences for different $k$ are
independent. 

Fix $J,K\ge1$ and consider only $j\le J$ and $k\le K$. Let, for $0\le k\le K$, 
\begin{equation}
\bB_{n-k}:= \floor{m^{K-k}B_{n-K}}  
\end{equation}
and let 
\begin{equation}\label{bW}
   \bW_{n-k,j}:=\sum_{i=1}^{\bB_{n-k}}\vecN\kki_j.
\end{equation}
Then by the central limit theorem, exactly as for \eqref{erika},
\begin{equation}\label{berika}
  \bB_{n-k}\qqw \bW_{n-k,j}\dto  \eta_j\kk,
\end{equation}
jointly for all $j\le J$ and $k\le K$; note that now,
if we condition on $B_{n-K}$, the \lhs{s} for
different $k$ are independent.
Furthermore, by \eqref{pyret} and \eqref{znkz},
$\bB_{n-k}/B_{n-k}\asto1$ for every $k$. 
Hence \eqref{berika} yields, jointly,
\begin{equation}\label{cerika}
  B_{n-k}\qqw \bW_{n-k,j}\dto  \eta_j\kk.
\end{equation}
Moreover, using \eqref{bW}, 
\begin{equation}
  \E\bigpar{(\bW_{n-k,j}-W_{n-k,j})^2\mid B_{n-k},\bB_{n-k}}
=|B_{n-k}-\bB_{n-k}|\Var N_j
\end{equation}
and, consequently,
for every fixed $j\ge0$,  $k\ge0$ and
$\eps>0$,
\begin{equation*}
  \PP\bigpar{|\bW_{n-k,j}-W_{n-k,j}|>\eps B_{n-k}\qq\mid B_{n-k},\bB_{n-k}}
\le |1-\bB_{n-k}/B_{n-k}|\gs_{jj}\eps\qww
\asto0.
\end{equation*}
Taking the expectation, we obtain by dominated convergence that
for every $j$ and $k$, 
$  \PP\bigpar{|\bW_{n-k,j}-W_{n-k,j}|>\eps B_{n-k}\qq}\to0$ for every $\eps>0$, 
and thus
\begin{equation}\label{cd}
B_{n-k}\qqw\bW_{n-k,j}-B_{n-k}\qqw W_{n-k,j}\pto0.  
\end{equation}
Combining \eqref{cerika} and \eqref{cd} yields
\begin{equation}\label{derika}
  B_{n-k}\qqw W_{n-k,j}\dto  \eta_j\kk,
\end{equation}
still jointly for all $j\le J$ and $k\le K$. 
The result follows by this and \eqref{per}, since $J$ and $K$ are arbitrary.
\end{proof}

\section{First proof of \refT{T1}}\label{SpfT1-A}


In this section we assume \refAA{} and also \ref{Aroots},
i.e., $\ggx>m\qqw$.
In other words, see \eqref{ggx},
each $z\in\gGG$ satisfies $|z|>m\qqw$.
Hence, we may decrease $r$ 
so that the disc $D_r$ contains
no roots of $\mmu(z)=1$ except $m\qw$, and still $r>m\qqw$. 
Thus, with $R:=1/r$ and assuming \refAA, we see that
$\ggx>m\qqw$ is equivalent to: 
\begin{Bnoenumerate}
\item[\ArootsR] 
  There exists $R$ with $1\le R<m\qq$ such that $\mmu(R\qw)<\infty$ and,
furthermore,  $\mmu(z)\neq 1$ for every complex $z\neq m\qw$ 
with $|z|<R\qw$.
\end{Bnoenumerate}
We fix an $R$ such that \ArootsR{} holds, 
and \ref{Ar} holds with $r=1/R$.
Note that $R$ may be chosen arbitrarily close to $m\qq$.
Furthermore, we fix $R_1$ with $R<R_1<m\qq$.
Then \ArootsR{} and \refL{LC} show that \eqref{lc} holds, i.e.,
$  \norm{T^n}_{\llr} =O\bigpar{ R_1^n}$.

\begin{lem}\label{LX1}
Assume \refAAA. If\/ $R<m\qq$, then
\begin{equation}\label{lx1a}
\E \norm{\vecx_n}_{\llr}^2 \le C m^n
\end{equation}
and thus
\begin{equation}\label{lx1b}
  \E X_{n,k}^2 \le CR^{-2k} m^n
\end{equation}
  for all $n,k\ge0$.
\end{lem}
\begin{proof}
By \eqref{ly}, \refL{LA}, \eqref{lc} and 
 Minkowski's inequality,
\begin{equation}\label{ele57}
  \begin{split}
\norm{\vecx_n}_{L^2(\llr)}    
&\le \sumkn\norm{W_{n-k}}_{L^2} \norm{T^k(\vec  v)}_{\llr}
\le C\sumkn m^{(n-k)/2}R_1^k
\\&
= Cm^{n/2}\sumk (R_1/m\qq)^k
= Cm^{n/2}.
  \end{split}
\end{equation}
This yields \eqref{lx1a}, and \eqref{lx1b} follows by \eqref{llr}.
\end{proof}

Define for convenience $W_{n,j}$ also for $n<0$ by $W_{-1,1}:=W_0$ and
$W_{n,j}=0$ for $n\le-1$ and $j\ge1$ with $(n,j)\neq(-1,1)$.
Then \eqref{wn1} holds also for $n\le0$, provided the sum is extended to
$\infty$, and \eqref{ly} can be written
\begin{equation}\label{ele1}
  \vecx_n=-\sumk \sumji W_{n-k-j,j}T^k(\vec v).
\end{equation}
For each finite $M$ define also the truncated sum
\begin{equation}\label{ele2}
  \vecx\nM:=-\sumkM \sumjM W_{n-k-j,j}T^k(\vec v).
\end{equation}
\refL{Leta} implies that for any fixed $M$, as \ntoo,
\begin{equation}\label{ele3}
Z_n\qqw  \vecx\nM\dto
-\sumkM \sumjM (1-m\qw)\qq m^{-(k+j)/2}\eta_j\kkx{k+j}T^k(\vec v)
\end{equation}
in $\llr$.
Furthermore, 
by \eqref{ele1}--\eqref{ele2}, Minkowski's
inequality,
\refL{LA} and \eqref{lc},
regarding $\vecx_n$ and $\vecx\nM$ as elements of $L^2(\llr)$, the space of
$\llr$-valued random variables with square integrable norm,
\begin{align}\label{ele5}
\norm{\vecx_n-\vecx\nM}_{L^2(\llr)}    
&\le \sum_{k>M \text{ or } j>M}\norm{W_{n-k-j,j}}_{L^2} \norm{T^k(\vec  v)}_{\llr}
\notag\\&
\le C\sum_{k>M \text{ or } j>M}r^{-j}m^{(n-k-j)/2}R_1^k
\notag\\&
= Cm^{n/2}\sum_{k>M \text{ or } j>M}(R/m\qq)^{j}(R_1/m\qq)^k.
\end{align}
Since the sum on the \rhs{} of \eqref{ele5} converges, it tends to 0 as
$M\to\infty$, and thus
$m^{-n/2} \bigpar{\vecx_n-\vecx\nM}\to 0$ in $L^2(\llr)$, and thus in
probability, uniformly in $n$.
Since $Z_n/m^n\asto \cZ>0$, see \eqref{Zoo}, $\sup_n m^n/Z_n$ is an \as{}
finite random variable; hence also
\begin{equation}
Z_n\qqw \bigpar{\vecx_n-\vecx\nM}
=
\Bigparfrac{m^n}{Z_n}\qq
m^{-n/2} \bigpar{\vecx_n-\vecx\nM}\pto 0
\end{equation}
as $M\to\infty$, uniformly in $n$.

Moreover, the \rhs{} of \eqref{ele3} converges as \Mtoo{} in $L^2(\llr)$, and
thus in distribution, since by \eqref{evk2}
\begin{equation}
  \E[(\eta_j\kk)^2] = \Var N_j \le \CCevk r^{-2j} =\CCevk R^{2j},
\end{equation}
and thus, using also \eqref{lc},
\begin{align}\label{elex}
\sumk \sumji  m^{-(k+j)/2}\norm{\eta_j\kkx{k+j}T^k(\vec v)}_{L^2(\llr)}
&
=
\sumk \sumji  m^{-(k+j)/2}\norm{\eta_j\kkx{k+j}}_{L^2}\norm{T^k(\vec v)}_{\llr} 
\notag\\&
\le \CC
\sumk \sumji  m^{-(k+j)/2}R^j R_1^k
<\infty. 
\end{align}
It follows, see \cite[Theorem 4.2]{Billingsley},
that \eqref{ele3} extends to $M=\infty$, 
i.e.,
\begin{equation}\label{ele33}
Z_n\qqw  \vecx_n\dto
-\sumk \sumji (1-m\qw)\qq m^{-(k+j)/2}\eta_j\kkx{k+j}T^k(\vec v)
\end{equation}
in $\llr$ as \ntoo.
The \rhs{} is obviously a Gaussian random vector in $\llr$, which we write
as $\vec \zeta=(\zeta_0,\zeta_1,\dots)$. 
Then \eqref{ele33} yields \eqref{t1a}.

It remains to calculate the covariances of $\zeta_k$.
Let $\vec a=(a_0,a_1,\dots)$ be a (real) vector with only finitely many non-zero
elements. 
Then, 
by \eqref{ele33},
\begin{equation}\label{eleison}
  \suml a_\ell\zeta_\ell = \innprod{\vec a ,\vec \zeta}
=
-(1-m\qw)\qq
\sumk \sumji  m^{-(k+j)/2}\eta_j\kkx{k+j}\innprod{T^k(\vec v),\vec a}
\end{equation}
with the sum converging absolutely in $L^2$ by \eqref{elex}.

By the definition of $\eta_j\kk$ in \eqref{coveta} and \refL{Leta},
\begin{equation}
  \begin{split}
  \Cov\Bigpar{m^{-k/2}\eta_i\kk,m^{-\ell/2}\eta_j\kkx{\ell}}
&=
m^{-(k+\ell)/2}\gd_{k,\ell}\gs_{ij}
=\oint_{|w|=m\qqw}\gs_{ij}w^k\bar w^\ell \frac{|\dw|}{2\pi m\qqw}.    
  \end{split}
\end{equation}
Hence, \eqref{eleison} yields
\begin{equation}\label{ele18}
  \begin{split}
&(1-m\qw)\qw\Var\bigpar{\innprod{\vec a,\vec\zeta} }
\\&\qquad
=\sumk\suml\sumii\sumji
\innprod{T^k(\vec v),\vec a}
\innprod{T^\ell(\vec v),\vec a}
\oint_{|w|=m\qqw}\gs_{ij}w^{k+i}\bar w^{\ell+j} \frac{|\dw|}{2\pi m\qqw}
\\&\qquad
=
\oint_{|w|=m\qqw}
\sumii\sumji
\gs_{ij}w^{i}\bar w^{j} 
\lrabs{\sumk w^k \innprod{T^k(\vec v),\vec a}}^2
\frac{|\dw|}{2\pi m\qqw}.    
  \end{split}
\raisetag{1.5\baselineskip}
\end{equation}
Furthermore, if $|w|=m\qqw$, then
$\sumk\norm{w^kT^k(\vec v)}_{\llr}<\infty$ by \eqref{lc},
and thus
\begin{equation}\label{ele11}
  \sumk w^k T^k(\vec v)
= (1-wT)\qw(\vec v).
\end{equation}
Let $\gl:=w\qw$, so $|\gl|=m\qq>R$. We use as in the proof of \refL{LB} the
standard isometry $\llr\to\hhr$, and let $f(z)\in\hhr$ be the function
corresponding to  
$ (1-wT)\qw(\vec v)=\gl(\gl-T)\qw(\vec v)$.
Thus, see \eqref{TH}--\eqref{winston},
\begin{equation}
  (\gl-z)f(z)-\chix(f)v(z) =(\gl-T)f(z)=\gl v(z)
\end{equation}
and thus, \cf{} \eqref{winston}--\eqref{anna},
\begin{equation}
  f(z)=b\frac{v(z)}{\gl-z}
\end{equation}
for a constant $b$ such that $b=\chix(f)+\gl$.
This yields by \eqref{bix}
\begin{equation}
b-\gl=\chix(f)=\frac{b}{m-\gl}\bigpar{(1-\gl)\mmu(\gl\qw)+m-1}
\end{equation}
with the solution
\begin{equation}
  b=\frac{\gl(m-\gl)}{(1-\gl)(1-\mmu(\gl\qw))}.
\end{equation}
Hence, using \eqref{vix}, for $|z|\le R$,
\begin{equation}
  \begin{split}
  f(z)
&=b\frac{v(z)}{\gl-z}
=\frac{\gl}{(1-\gl)(1-\mmu(\gl\qw))}\Bigpar{\frac{\gl}{\gl-z}-\frac{m}{m-z}}
\\&
=\frac{\gl}{(1-\gl)(1-\mmu(\gl\qw))}\suml (\gl^{-\ell}-m^{-\ell})z^\ell  . 
\\&
=\frac{1}{(w-1)(1-\mmu(w))}\suml (w^{\ell}-m^{-\ell})z^\ell  . 
  \end{split}
\end{equation}
Thus,
$(1-wT)\qw(\vec v)=\bigpar{\xpar{(w-1)(1-\mmu(w))}\qw(w^\ell-m^{-\ell})}_\ell$
and, using \eqref{ele11},
\begin{equation}\label{qkq}
  \begin{split}
  \sumk w^k \innprod{T^k(\vec v),\vec a}
&=
\innprod{(1-wT)\qw(\vec v),\vec a}
=
\frac{1}{(w-1)(1-\mmu(w))}\suml a_\ell(w^\ell-m^{-\ell}).    
  \end{split}
\end{equation}
Hence \eqref{t1b} follows from \eqref{ele18}.

Finally, by \eqref{t1b}, the variable $\zeta_k$ is degenerate only if $\gS(z)=0$
for every $z$ with $|z|=m\qqw$, and thus, by \eqref{tc},
$\mXi(z)=\mmu(z)$ \as{} for every such $z$, which by \eqref{mmu}--\eqref{mXi}
implies $N_k=\mu_k$ \as{} for every $k$.
\qed

\section{A martingale}\label{Smart}

In the remaining sections, we  let $R:=r\qw<m\qq$,
where $r$ is as in \ref{Ar}.
(We may assume that $R$ is arbitrarily close to $m\qq$ by decreasing $r$.)
We consider as above the operator $T$ on $\llr$.


Fix a real vector $\vec a\in\ell^2_{R\qw}$ (for example any finite real
vector), and write  
\begin{equation}\label{alphak}
\alpha_k=\alpha_k(\veca):=
  \innprod{T^k(\vec v),\vec  a}.
\end{equation}
Then \eqref{ly} and \eqref{wn1} yield
\begin{equation}\label{ele19}
 \innprod{\vecx_n,\vec a}
=-\sumk \sumji W_{n-k-j,j}\alpha_k 
=-\sum_{\ell=0}^{n} \sum_{j=1}^{n-\ell} W_{\ell,j}
\alpha_{n-j-\ell}
\end{equation}

Define
\begin{align}
  \gDM_{n,\ell}&:= \sum_{j=1}^{n-\ell} \alpha_{n-j-\ell}W_{\ell,j}, \label{gDM}
\\
  M_{n,k}&:=\sum_{\ell=0}^k\gDM_{n,\ell}.  \label{mnk}
\end{align}
Then \eqref{mgw} shows that 
$\E\bigpar{\gDM_{n,\ell}\mid \FF_{\ell-1}}=0$, and thus $(M_{n,k})_{k=0}^n$ is a
martingale \wrt{} $(\FF_k)_k$.
Furthermore, by \eqref{ele19},
\begin{equation}\label{vende}
   \innprod{\vecx_n,\vec a}=-M_{n,n}.
\end{equation}
Conditioned on $\FF_{\ell-1}$, the vector $(W_{\ell,j})_j$ is the sum
of $B_\ell$ independent copies of $\vecN-\E\vecN$, 
where $\vecN=(N_j)\xoo$,
and thus, recalling \eqref{gsjk},
\begin{equation}\label{unl}
  \begin{split}
Q_{n,l}&:=   \E \bigpar{(\gD M_{n,\ell})^2\mid\FF_{\ell-1}}
=B_\ell \Var\biggpar{\sum_{j=1}^{n-\ell} \alpha_{n-\ell-j}N_j}
\\&\phantom:
=B_\ell \sum_{i,j=1}^{n-\ell} \gs_{ij}
\alpha_{n-\ell-i}\alpha_{n-\ell-j}.
  \end{split}
\end{equation}
The conditional quadratic variation of the martingale $(M_{n,k})_k$ is thus
\begin{equation}\label{Vn}
  \begin{split}
V_n:=\sum_{\ell=0}^n Q_{n,\ell} 
&= \sum_{\ell=0}^n B_\ell \sum_{i,j=1}^{n-\ell} \gs_{ij}\alpha_{n-\ell-i}\alpha_{n-\ell-j}
= \sum_{\ell=0}^n B_{n-\ell} 
\sum_{i,j=1}^{\ell} \gs_{ij}\alpha_{\ell-i}\alpha_{\ell-j}.
  \end{split}
\end{equation}

By \eqref{mXi}, $N_k\le r^{-k}\mXi(r)$, and thus 
by \eqref{gsjk} and the Cauchy--Schwarz inequality, 
\begin{equation}\label{gsb}
  |\gs_{ij}|\le  r^{-i-j} \E \mXi(r)^2=CR^{i+j}.
\end{equation}


\section{Second proof of \refT{T1}}\label{SpfT1-B}

As said earlier,
we give here another proof of \refT{T1}, 
based on a martingale central limit theorem. and the martingale in
\refS{Smart}.
The main reason is that the new proof with small
modifications also applies to \refT{T2}, see \refS{SpfT2}, and we prefer to
present it first for \refT{T1}. 
(The proof in \refS{SpfT1-A} does not seem to extend easily to \refT{T2}.)

Let $R$ and $R_1$ be as in \refS{SpfT1-A}.
Then, \eqref{alphak} and \eqref{lc} show that, for a fixed $\vec a$,
with $C=C(\vec a)$,
\begin{equation}\label{alphabit}
  |\alpha_k|\le C R_1^k. 
\end{equation}
Consequently, by \eqref{unl},  \eqref{gsb} and \eqref{alphabit}, 
since  $R/R_1<1$,
\begin{equation}\label{ub}
  \begin{split}
\frac{Q_{n,\ell}}{B_\ell}
= \sum_{i,j=1}^{n-\ell} \gs_{ij}\alpha_{n-\ell-i}\alpha_{n-\ell-j}
\le C \sum_{i,j=1}^{\infty} R^{i+j}R_1^{2(n-\ell)-i-j}
\le C R_1^{2(n-\ell)}.
  \end{split}
\end{equation}
Hence, by \eqref{Vn}, \eqref{unl}, \eqref{bn} and \eqref{znkz},
using dominated convergence justified by \eqref{ub} and $R_1^2/m<1$,
\begin{equation}\label{VZ}
  \begin{split}
\frac{  V_n}{Z_n }
&=
\sum_{\ell=0}^n \frac{B_{n-\ell}}{Z_n}\frac{Q_{n,n-\ell}}{B_{n-\ell}}
=
\sum_{\ell=0}^n \frac{Z_{n-\ell}-Z_{n-\ell-1}}{Z_n} 
\sum_{i,j=1}^{\ell} \gs_{ij}\alpha_{\ell-i}\alpha_{\ell-j}
\\&
\asto
\gss(\vec a):=
\sum_{\ell=0}^\infty \bigpar{m^{-\ell}-m^{-\ell-1}}
\sum_{i,j=1}^{\ell} \gs_{ij}\alpha_{\ell-i}\alpha_{\ell-j}
  \end{split}
\end{equation}

We cannot use a martingale central limit theorem directly for the martingale 
$(M_{n,k})_k$ defined in \eqref{mnk}, because the calculations above show
that most of the conditional quadratic variation $V_n$ comes from a few
terms (the last ones), \cf{} \refR{RLY}.
We thus introduce another martingale.

Number the individuals $1,2,\dots$ in order of birth, with arbitrary order
at ties, and let $\GG_\ell$ be the $\gs$-field generated by the life
histories of individuals $1,\dots,\ell$.
Each $Z_n$ is a stopping time with respect to $(\GG_\ell)_\ell$, and
$\GG_{Z_n}=\FF_n$.

We refine the martingale $(M_{n,k})_k$ by adding the contribution from each
individual separately. Let $\tau_i$ denote the birth time of $i$, and
$\Ni k$ the copy of $N_k$ for $i$ (i.e., the number of children $i$
gets at age $k$). Let
\begin{align}
  \gDMx_{n,i}&:=\sum_{j=1}^{n-\tau_i}\alpha_{n-\tau_i-j}\bigpar{\Ni{j}-\mu_j},
\label{gDMx}
\\
\Mx_{n,k}&:=\sum_{i=1}^k\gDMx_{n,i}.
\label{Mx}
\end{align}
Then $(\Mx_{n,k})_k$ is a $(\GG_k)_k$-martingale 
with $\Mx_{n,\infty}=\Mx_{n,Z_n}=M_{n,n}=-\innprod{\vecx_n,\vec a}$,
see \eqref{gDM}--\eqref{vende}, and 
the conditional quadratic variation
\begin{equation}\label{Vx}
\Vx_n:=\sum_i  \E\bigpar{(\gDMx_{n,i})^2\mid\GG_{i-1}} = V_{n}
\end{equation}
given by \eqref{Vn}.
Moreover, by \eqref{gDMx} and \eqref{alphabit},
\begin{equation}\label{holm}
  \begin{split}
\bigabs{\gDMx_{n,i}} 
&\le C \sumj R_1^{n-\tau_i-j}\bigpar{\Ni{j}+\mu_j}
= C R_1^{n-\tau_i}\bigpar{\mXi\xii(R_1\qw)+\mmu(R_1\qw)}.
  \end{split}
\end{equation}
Define the random variable $U:=\mXi(R_1\qw)+\mmu(R_1\qw)$. 
Then $\E U^2<\infty$ by \ref{Ar}, since $R_1\qw<r$.
It follows from \eqref{holm} that for some $c>0$ and every $\eps>0$,
defining $h(x):=\E\bigpar{U^2\ett{U> c x}}$,
\begin{multline}
    \E\bigpar{\bigabs{\gDMx_{n,i}}^2\ett{\bigabs{\gDMx_{n,i}}>\eps}\mid\GG_{i-1}} 
\le C R_1^{2(n-\tau_i)}\E\bigpar{U^2\ett{U> c \eps R_1^{\tau_i-n}}}
\\
= C R_1^{2(n-\tau_i)}h\bigpar{\eps R_1^{\tau_i-n}}
\le C R_1^{2(n-\tau_i)}h\bigpar{\eps R_1^{-n}},\quad
\end{multline}
Thus,
\begin{equation}
\label{svea}
\sum_i
\E\bigpar{\bigabs{\gDMx_{n,i}}^2\ett{\bigabs{\gDMx_{n,i}}>\eps}\mid\GG_{i-1}} 
\le 
C\sum_{k=0}^n B_k R_1^{2(n-k)}  h\bigpar{\eps R_1^{-n}}  . 
\end{equation}
Finally, we normalize $\Mx_{n,k}$ and define $\My_{n,k}:=m^{-n/2}\Mx_{n,k}$;
this yields a martingale $(\My_{n,k})_k$ with conditional quadratic
variation
\begin{equation}\label{Vy}
\Vy_n:=\sum_i  \E\bigpar{(\gDMy_{n,i})^2\mid\GG_{i-1}} = m^{-n}\Vx_{n}
\asto
\gss(\vec a)\cZ,
\end{equation}
by \eqref{Vx}, \eqref{VZ} and \eqref{Zoo}. 
Furthermore, by \eqref{svea}, 
\begin{equation}
\label{gota}
  \begin{split}
\sum_i
\E\bigpar{\bigabs{\gDMy_{n,i}}^2\ett{\bigabs{\gDMy_{n,i}}>\eps}\mid\GG_{i-1}} 
\le 
Ch\bigpar{\eps m^{n/2}R_1^{-n}}m^{-n}  \sum_{k=0}^n B_k R_1^{2(n-k)} , 
  \end{split}
\end{equation}
which tends to 0 \as{} as \ntoo, because $(m\qq R_1\qw)^{n}\to\infty$ 
and consequently $h\bigpar{\eps m^{n/2}R_1^{-n}}\to0$, and
\begin{equation}
m^{-n} \sum_{k=0}^n B_k R_1^{2(n-k)}
= m^{-n} \sum_{k=0}^n B_{n-k} R_1^{2k}
=  \sum_{k=0}^n \frac{B_{n-k}}{m^{n-k}} \Bigparfrac{R_1^2}{m}^{k}=\Oas(1),
\end{equation}
by \eqref{Zoo} and $R_1^2<m$.

The martingales $(\My_{n,i})_i$ thus satisfy a conditional Lindeberg
condition, which together with \eqref{Vy} implies,
by \cite[Corollary 3.2]{HH}, that, using \eqref{Vx},
\begin{equation}\label{martin}
  M_{n,n}/V_n\qq
=
  \Mx_{n,Z_n}/{\Vx_n}\qq
=
  \My_{n,Z_n}/\Vy_n\qq
\dto N(0,1)
\end{equation}
as \ntoo; furthermore, the limit is mixing.
(The fact that we here sum the martingale differences to a stopping time
$Z_n$ instead of a deterministic $k_n$ as in \cite{HH} makes no difference.) 
By \eqref{vende}
and \eqref{VZ}, this yields
\begin{equation}\label{martina}
\innprod{\vecx_n,\vec a}/Z_n\qq\dto N\bigpar{0,\gss(\vec a)}.  
\end{equation}

We can evaluate the asymptotic variance $\gss(\vec a)$ given in \eqref{VZ} by 
\begin{equation}\label{m2}
  \begin{split}
    \frac{\gss(\vec a)}{1-m\qw}
&=
\sum_{\ell=0}^\infty m^{-\ell} \sum_{i,j=1}^{\ell} \gs_{ij}\alpha_{\ell-i}\alpha_{\ell-j}
\\&
=
\sum_{k,p=0}^\infty  \sum_{i,j=1}^{\ell}
\gs_{ij}\alpha_{k}\alpha_{p}\ett{i+k=j+p}m^{-i-k}
\\&
=\sum_{k,p,i,j}\gs_{ij}\alpha_k\alpha_p
\oint_{|z|=m\qqw}z^{i+k}\bz^{j+p}
\frac{|\dz|}{2\pi m\qqw}
\\&
=
\oint_{|z|=m\qqw}\Bigabs{\sum_k \alpha_kz^{k}}^2\sum_{i,j}\gs_{ij}z^i\bz^j
\frac{|\dz|}{2\pi m\qqw}.
  \end{split}
\end{equation}
Furthermore,
for $|z|=m\qqw$ (and any $z$ with $|z|<R\qw=r$ and $\mmu(z)\neq1$),
by \eqref{alphak} and \refL{Lres} with $\gl=z\qw$,
\begin{equation}\label{m3}
  \begin{split}
\sum_{k=0}^\infty \alpha_kz^{k}
& = \Biginnprod{\sumk z^kT^k(\vecv),\vec a}    
=\biginnprod{(1-zT)\qw(\vecv),\vec a}
\\&
=\frac{1}{(z-1)(1-\mmu(z))}\sum_\ell a_\ell\bigpar{z^\ell-m^{-\ell}}.
  \end{split}
\end{equation}
By \eqref{m2}--\eqref{m3}, $\gss(\vec a)$ equals the \rhs{} in \eqref{t1b}.
Thus, \eqref{martina} 
shows convergence as in \eqref{t1a} for any finite linear combination of 
$Z_n\qqw X_{n,k}$, and thus joint convergence in \eqref{t1a} by the
Cram\'er--Wold device.

Convergence in $L^2(\llr)$ follows from this and \refL{LX1} (with a slightly
increased $R$) by a standard truncation argument; we omit the details.

By \eqref{t1b}, the variable $\zeta_k$ is degenerate only if $\gS(z)=0$
for every $z$ with $|z|=m\qqw$, and thus, by \eqref{tc},
$\mXi(z)=\mmu(z)$ \as{} for every such $z$, which by \eqref{mmu}--\eqref{mXi}
implies $N_k=\mu_k$ \as{} for every $k$.
\qed

\section{Proof of \refT{T2}}\label{SpfT2}
We assume in this section that $\ggx= m\qqw$ and that \eqref{AT3} holds.
By \refL{LB}, the spectral radius $\rhox(T)=\ggx\qw= m\qq$.
\refL{LD} applies with $\ggx=m\qqw$, and thus $\tR<m\qq$;
we may assume $\tR>R$.

Fix as in \refS{Smart} a real vector $\vec a\in \ell^2_{R\qw}$,
and define,
 using \eqref{vecvi},  
\begin{equation}\label{gak}
\gb_i=  \gb_i(\veca):=\innprod{P_i(\vecv),\veca}=\innprod{\vecv_i,\veca}
=\frac{1}{\gam_i(\gam_i-1)\mmu'(\gam_i)}\sumk a_k\bigpar{\gam_i^{k}-m^{-k}}.
\end{equation}
Then, 
by \eqref{alphak} and \refL{LD},
\begin{equation}\label{alphakk}
  \alpha_k=O\bigpar{\tR^k}+\sumiq\gl_i^k\innprod{P_i(\vecv),\veca}
=
\sumiq \gb_i\gl_i^k+ O\bigpar{\tR^k}
= O\bigpar{m^{k/2}}.
\end{equation}
Furthermore, the $O$'s in \eqref{alphakk} hold
uniformly in all $\veca$ with $\norm{\veca}_\llrqw\le1$,
as does every $O$ in this section.

Define also, for $p,t=1,\dots,q$,
\begin{equation}
  \gsx_{pt}:=\sum_{i,j=1}^{\infty} \gs_{ij}\gl_p^{-i}\gl_t^{-j},
\end{equation}
and note that, using \eqref{gsb}, $|\gl_p|=m\qq$ and $R<m\qq$,
\begin{equation}\label{zyx}
\sum_{i,j=1}^{\ell} \gs_{ij}\gl_p^{-i}\gl_t^{-j}
=     \gsx_{pt}
+O\Bigpar{\sum_{i>\ell, j\ge1} R^{i+j} (m\qq)^{-i-j}}
=\gsx_{pt}+O\bigpar{(R/m\qq)^\ell}.
\end{equation}

Let
\begin{equation}\label{sl}
  \begin{split}
  s_\ell:= \sum_{i,j=1}^{\ell} \gs_{ij}\alpha_{\ell-i}\alpha_{\ell-j}.    
  \end{split}
\end{equation}
Then, by \eqref{alphakk} and symmetry, using again \eqref{gsb} and
$|\gl_p|=m\qq$,
and \eqref{zyx},
\begin{equation}\label{slx}
  \begin{split}
  s_\ell&:= 
\sum_{i,j=1}^{\ell}  \gs_{ij}
\sumpq\sumtq\gb_p\gl_p^{\ell-i}\gb_t\gl_t^{\ell-j}
+
O\Bigpar{\sum_{i,j=1}^{\ell}  R^{i+j} m^{(\ell-i)/2} \tR^{\ell-j}}
\\&\phantom:
=\sumpq\sumtq\gb_p\gb_t \gl_p^{\ell}\gl_t^{\ell}
\sum_{i,j=1}^{\ell}  \gs_{ij}\gl_p^{-i}\gl_t^{-j}
+O\bigpar{(m\qq\tR)^{\ell}}
\\&\phantom:
=\sumpq\sumtq\gb_p\gb_t \gl_p^{\ell}\gl_t^{\ell}
\gsx_{pt}
+O\bigpar{(m\qq\tR)^{\ell}}.
  \end{split}
\end{equation}
In particular,
\begin{equation}
  \label{sln}
s_\ell=O\bigpar{m^\ell}.
\end{equation}
 It follows by \eqref{Vn}, \eqref{sl}, \eqref{Zoo}, \eqref{sln} and \eqref{slx} 
that, a.s., 
\begin{equation}\label{po}
  \begin{split}
\frac{V_n}{B_n}&
=\sumln\frac{B_{n-\ell}}{B_n}s_\ell    
=\sumln m^{-\ell}\bigpar{1+o(1)+\Oas(1)\ett{n-\ell<\log n}}s_\ell    
\\&
=\sumln m^{-\ell}s_\ell +o(n)
=\sumln m^{-\ell}\sumpq\sumtq\gb_p\gb_t \gl_p^{\ell}\gl_t^{\ell}\gsx_{pt}
 +o(n)
\\&
=
\sumpq\sumtq\gb_p\gb_t \gsx_{pt}
\sumln 
\Bigparfrac{\gl_p\gl_t}{m}^{\ell}
 +o(n).
  \end{split}
\end{equation}
Recall that $|\gl_p|=|\gl_t|=m\qq$, so $|\gl_p\gl_t/m|=1$. 
Hence, if $\gl_t=\bar\gl_p$, then 
$\sumln \bigpar{\xfrac{\gl_p\gl_t}{m}}^{\ell}=n+1$,
while if $\gl_t\neq\bar\gl_p$, then 
$\sumln \bigpar{\xfrac{\gl_p\gl_t}{m}}^{\ell}=O(1)$.
Consequently,
\eqref{po} yields, since $B_n/Z_n\asto1-m\qw$ by \eqref{bn} and \eqref{Zoo},
\begin{equation}\label{mote}
  \begin{split}
\frac{V_n}{nZ_n}
\asto
\gss(\veca)
&:= 
\ettmx
\sumpq\sumtq\gb_p\gb_t \gsx_{pt}\ett{\gl_t=\bar\gl_p}
\\&\phantom:
= \ettmx\sumpq|\gb_p|^2\sum_{i,j=1}^\infty\gs_{ij}\gl_p^{-i}\bar\gl_p^{-j}
\\&\phantom:
= \ettmx\sumpq|\gb_p|^2\gS(\gam_p).
  \end{split}
\end{equation}

We refine the martingale $(M_{n,k})_k$ to $(\Mx_{n,k})_k$ as in
\refS{SpfT1-B}, but this time we normalize it to 
$\My_{n,k}:=(nm^n)\qqw\Mx_{n,k}$.
It follows from \eqref{mote} and \eqref{Zoo}
that the conditional quadratic variation
$\Vy_n=V_n/(nm^n)\asto \gss(\veca)\cZ$, \ie, \eqref{Vy} holds also in the
present case. 
Furthermore, if we now let $R_1:=m\qq$, then \eqref{alphabit} and
\eqref{holm}--\eqref{svea}
hold, and it follows that \eqref{gota} is modified to
\begin{align}
\label{gotaII}
\sum_i
\E\bigpar{\bigabs{\gDMy_{n,i}}^2\ett{\bigabs{\gDMy_{n,i}}>\eps}\mid\GG_{i-1}} 
&\le 
\notag
Ch\bigpar{\eps n^{1/2}}\frac{1}{nm^{n}}  \sum_{k=0}^n B_k m^{n-k} 
\\&
=\Oas\bigpar{h\bigpar{\eps n^{1/2}}}\asto0.
\end{align}
Hence the conditional Lindeberg condition holds in the present case too,
and \eqref{martin} follows again by \cite[Corollary 3.2]{HH}, which now by 
\eqref{mote} and \eqref{vende} yields (mixing)
\begin{equation}\label{martinb}
\innprod{\vecx_n,\vec a}/(nZ_n)\qq\dto N\bigpar{0,\gss(\vec a)}.  
\end{equation}
By \eqref{mote} and \eqref{gak}, this proves \eqref{t2a}--\eqref{t2b}.

By \eqref{t2b}, the variable $\zeta_k$ is degenerate only if $\gS(\gam_p)=0$
for every $p$, and thus, by \eqref{tc}, $\mXi(\gam_p)=\mmu(\gam_p)$ a.s.

As in \refS{SpfT1-B}, convergence in $L^2(\llr)$ follows 
by a standard truncation argument, now using the following lemma
(with an increased $R$); we omit the details.
\qed

\begin{lem}\label{LX2}
Assume \refAA, $\ggx=m\qqw$ and \eqref{AT3}. 
If\/ $R<m\qq$, then
\begin{equation}\label{lx2a}
\E \norm{\vecx_n}_{\llr}^2 \le C nm^n
\end{equation}
and 
\begin{equation}\label{lx2b}
  \E X_{n,k}^2 \le Cn m^nR^{-2k}
\end{equation}
  for all $n,k\ge0$.
\end{lem}

\begin{proof}
By \eqref{vende}, \eqref{Vn}, \eqref{sl}, \eqref{EZ} and \eqref{sln},
\begin{equation}
  \E\innprod{\vecx_n,\veca}^2 = \E V_n = \E \sumln B_{n-\ell}s_\ell
\le C nm^n,
\end{equation}
uniformly for $\norm{\veca}_{\llrqw}\le1$.
Taking $\veca=R^k(\gd_{kj})_j$, we  obtain \eqref{lx2b}.

Finally, applying \eqref{lx2b} with $R$ replaced by some $R'$ with $R<R'<m\qq$,
\begin{equation}
\E \norm{\vecx_n}_{\llr}^2 =\sumk R^{2k} \E X_{n,k}^2\le C nm^n\sumk
(R/R')^{2k}
=Cnm^n.
\end{equation}
\end{proof}

\section{Proof of \refT{T3}}\label{SpfT3}
Assume now that $\ggx< m\qqw$.
By \refL{LB}, the spectral radius $\rhox(T)=\ggx\qw\ge m\qq$.
We apply \refL{LD}, assuming as we may that $\tR>m\qq$. (Otherwise
we increase $\tR$, keeping $\tR<\ggx\qw$.)
Hence, by \eqref{toa},
\begin{equation}\label{xa}
  T^k(\vecv) 
=
T_0^k(\vecv)+\sumiq \gl_i^k P_i(\vecv)
=
T_0^k(\vecv)+\sumiq \gl_i^k \vecv_i.
\end{equation}
Thus, by \eqref{ly},
\begin{equation}\label{xb}
  \vecx_n = -\sumkn W_k(TP_0)^{n-k}(\vecv) 
- \sumiq \sum_{k=0}^n \gl_i^{n-k}  W_k \vecv_i.
\end{equation}
Let, recalling \eqref{wn1},
\begin{equation}\label{ui}
  \qU_i:= -\sumk \gamma_i^{k}W_k
=-\sumk \gl_i^{-k}W_k
=-\suml\sumji \gl_i^{-\ell-j}W_{\ell,j},
\end{equation}
noting that by \refL{LA} and
$|\gamma_i|=\ggx<m\qqw$,
the sum converges in $L^2$  and
\begin{equation}\label{u2}
\Bignorm{\qU_i+\sumkn \gl_i^{-k}W_k}_2
\le\sum_{k=n+1}^\infty C  |\gl_i|^{-k}m^{k/2}
\le C\bigpar{\ggx m\qq}^n.
\end{equation}
Furthermore,
by \refL{LA} and \eqref{ton}, since $\tR>m\qq$, 
\begin{equation}\label{xc}
  \begin{split}
  \Bignorm{\sumkn W_k(TP_0)^{n-k}(\vecv)}_{L^2(\llr)}
&\le \sumkn \norm{W_k}_2\cdot   \norm{(TP_0)^{n-k}(\vecv)}_{\llr}
\\&
\le C \sumkn m^{k/2}\tR^{n-k}
\le C \tR^n.    
  \end{split}
\end{equation}

By \eqref{xb}, \eqref{xc}, \eqref{u2},
defining $U_i:=\bigpar{\gam_i(\gam_i-1)\mmu'(\gam_i)}\qw\qU_i$ so 
$\qU_i\vecv_i=U_i\vecu_i$
by \eqref{vecvi},
\begin{equation}\label{xx}
  \begin{split}
\Bignorm{\ggx^n \vecx_n - \sumiq \bigpar{\gl_i/|\gl_i|}^n U_i \vecu_i}_{L^2(\llr)}
&\le C \ggx^n \tR^n 
+ \sumiq\Bignorm{\sumkn\gl_i^{-k}W_k \vecv_i+\qU_i\vecv_i}_{L^2(\llr)}
\\&
\le  C (\ggx \tR)^n
+ C ( \ggx m\qq)^n
\le  C ( \ggx \tR)^n.
  \end{split}
\raisetag\baselineskip
\end{equation}
Since $\ggx \tR<1$, this shows convergence in \eqref{t3} in $L^2(\llr)$;
furthermore,
convergence \as{} follows by \eqref{xx} and
the Borel--Cantelli lemma.

We have $\E U_i=\E \qU_i=0$ by \eqref{ui} since $\E W_k=0$ by
\eqref{mgw}--\eqref{wn1}.
Furthermore,  $W_{0,k}=B_{0,k}-\mu_k=N_k-\mu_k$, while
$\E\bigpar{ W_{n,k}\mid\FF_0}=0$ for $n\ge1$ by \eqref{mgw}; 
hence by \eqref{wn1},
$\E\bigpar{W_{n}\mid\FF_0}=W_{0,n}=N_n-\mu_n$, and thus
\begin{equation}
\E\bigpar{ \qU_i\mid\FF_0}=-\sumk \gamma_i^k(N_k-\mu_k)
=-\mXi(\gamma_i)+\mmu(\gamma_i).
\end{equation}
Hence, $U_i$ is degenerate only if $\mXi(\gamma_i)$ is so.
\qed

\section{A stochastic integral calculus}\label{Sstoch}

The limit variables $\zeta_k$ in Theorems \ref{T1} and \ref{T2} can be
interpreted as stochastic integrals of certain functions (``\emph{symbols}'');
which gives a useful symbolic calculus. 
There are also some partial related results for \refT{T3}.

We consider the three cases in Theorems \ref{T1}--\ref{T3} separately.

\subsection{The case $\ggx>m\qqw$}
Assume throughout this subsection  
that \refT{T1}  applies; in particular that $\ggx> m\qqw$.

Let $\nu$ be the finite measure on the circle $|z|=m\qqw$ given by
\begin{equation}\label{nu}
\dd\nu(z):=
  \frac{m-1}m
|1-z|\qww\,|1-\mmu(z)|\qww	
\gS(z)\frac{|\dz|}{2\pi m\qqw},
\end{equation}
and consider an isomorphism $\SI: L^2(\nu)\to \cH$ of the  Hilbert
space $L^2(\nu)$ into a Gaussian Hilbert space $\cH$, \ie, a Hilbert space
of Gaussian random variables;  
$\SI$ can be interpreted as a stochastic integral, 
see \cite[Section VII.2]{SJIII}.
We let here $L^2(\nu)$ be the space of complex square-integrable functions, but
regard it as a real Hilbert space with the inner product 
$\innprod{f,g}_\nu:=\Re\int f\bar g\dd\nu$.
Then \eqref{t1a}--\eqref{t1b} 
can be stated as
\begin{equation}\label{si1}
      Z_n\qqw X_{n,k}
\dto \zeta_k
:=\SI\bigpar{z^k-m^{-k}},
\end{equation}
jointly for all $k\ge0$.
This yields a convenient calculus for joint limits.

\begin{ex}\label{ESI1}
 Let $k,\ell\ge0$.
Then, by \eqref{ynk},
\begin{equation}
  X_{n-\ell,k}=X_{n,k+\ell}-m^{-k}X_{n,\ell}
\end{equation}
and thus, recalling \eqref{znkz},
jointly for all $k,\ell\ge0$,
\begin{equation}\label{esi1a}
  \begin{split}
  Z_{n-\ell}\qqw X_{n-\ell,k}
&\dto
  m^{\ell/2}\bigpar{\zeta_{k+\ell}-m^{-k}\zeta_\ell}
=m^{\ell/2}\SI\bigpar{z^{k+\ell}-m^{-k}z^\ell}
\\&
=\SI\bigpar{(zm\qq)^\ell(z^{k}-m^{-k})}.    
  \end{split}
\end{equation}
Denoting this limit by $\zeta_k\llll$, we have of course
$\zeta_k\llll\eqd\zeta_k$, which corresponds to the fact that
$|zm\qq|^\ell=1$ on the support of $\nu$. More interesting is the joint
convergence
$(  Z_{n}\qqw X_{n,k},  Z_{n-\ell}\qqw X_{n-\ell,k})\dto(\zeta_k,\zeta_k\llll)$,
with covariance
\begin{equation}\label{eq85}
  \begin{split}
\Cov\bigpar{\zeta_k,\zeta_k\llll}
&=  
\innprod{z^{k}-m^{-k},(zm\qq)^\ell(z^{k}-m^{-k})}_\nu
\\&
=
\Re \int_{|z|=m\qqw} (zm\qq)^\ell|z^{k}-m^{-k}|^2  \dd\nu.
\end{split}
\end{equation}

The measure $\nu$ is by \eqref{nu}
absolutely continuous on the circle $|z|=m\qqw$.
With the change of variables $z=m\qqw e^{\ii \theta}$, we have
$(zm\qq)^\ell=e^{\ii\ell\theta}$ and the Riemann--Lebesgue lemma shows that
$\Cov\bigpar{\zeta_k,\zeta_k\llll}\to0$ as $\ell\to\infty$, for fixed every $k$.
Roughly speaking, $X_{n-\ell,k}$ and $X_{n,k}$ are thus essentially
uncorrelated when $\ell$ is large, which justifies the claim in
\refS{Smain}
that there is
only a short-range dependence in this case.
\end{ex}

\begin{ex}
  \label{ESI2}
We can define $X_{n,k}$ by \eqref{ynk} also for $k<0$.
Then, the calculations in \refE{ESI1} apply to any $\ell\ge0$ and any
$k\ge-\ell$. Hence, replacing $n$ by $n+\ell$ in \eqref{esi1a},
for any fixed $\ell$,
\begin{equation}\label{esi2}
  Z_{n}\qqw X_{n,k}
\dto
\SI\bigpar{(zm\qq)^\ell(z^{k}-m^{-k})}    
\end{equation}
jointly for all $k\ge-\ell$.
Since the factor $(zm\qq)^\ell$ does not depend on $k$ and has absolute value 1,
this means (by changing the isomorphism $\SI$) 
that \eqref{si1} holds jointly for all
$k\ge-\ell$.
Since $\ell$ is arbitrary, this means that \eqref{si1}
holds jointly for all $k\in\bbZ$. 
Hence,  \eqref{t1a}--\eqref{t1b} 
extend to all $k\in\bbZ$, as claimed in \refR{Rk<0}.
\end{ex}

\begin{ex}
  \label{ESI3}
We have, by \eqref{ynk},
\begin{equation}
  m^{-\ellj}Z_{n+\ellj}-m^{-\ellj-1}Z_{n+\ellj+1}=m^{-\ellj}X_{n+\ellj+1,1}.
\end{equation}
Hence, by \refL{LX1}, for $j\ge0$,
\begin{equation}\label{jw1}
\norm{m^{-\ellj}Z_{n+\ellj}-m^{-\ellj-1}Z_{n+\ellj+1}}_2 
\le  C m^{-\ellj + (n+\ellj+1)/2}
=  C m^{n/2-\ellj/2}.
\end{equation}
Summing \eqref{jw1} for $\ellj\ge\ell$ we obtain, recalling \eqref{Zoo},
\begin{equation}
\norm{m^{-\ell}Z_{n+\ell}-m^{n}\cZ}_2 
\le  C m^{n/2-\ell/2}
\end{equation}
for $n\ge1$ and $\ell\ge0$.
Hence,  as $\ell\to\infty$,
$m^{-n/2}\bigpar{m^{-\ell}Z_{n+\ell}-m^{n}\cZ}\to 0$ in $L^2$, and thus in
probability, uniformly in $n$. 
Since $Z_n/m^n\asto\cZ>0$, 
and thus $\sup_n m^n/Z_n<\infty $ a.s.,
it follows that, still
uniformly in $n$,
\begin{equation}
  \label{jw2}
Z_n\qqw\bigpar{m^{-\ell}Z_{n+\ell}-m^{n}\cZ}\pto 0,
\qquad \ell\to\infty.
\end{equation} 

Define the random variables
\begin{equation}\label{jv3}
  Y_{n,\ell}:= Z_n\qqw\bigpar{Z_n-m^{-\ell}Z_{n+\ell}}
=-Z_n\qqw m^{-\ell} X_{n,-\ell},
\qquad \ell\ge0.
\end{equation} 
Then, by \eqref{si1} and \refE{ESI2},  for every fixed $\ell$,
\begin{equation}\label{vnell}
  Y_{n,\ell}\dto- m^{-\ell}\zeta_{-\ell}=\SI\bigpar{1-m^{-\ell}z^{-\ell}},
\qquad \ntoo.
\end{equation}
Furthermore, by \eqref{jw2}, $Y_{n,\ell}\pto Z_n\qqw\bigpar{Z_n-m^n\cZ}$ as
$\ell\to\infty$, uniformly in $n$.
Finally, $|mz|=m\qq>1$ on the support of $\nu$, and thus $1-(mz)^{-\ell}\to
1$ in $L^2(\nu)$ as $\ell\to\infty$; hence
$\SI\bigpar{1-m^{-\ell}z^{-\ell}}\to\SI(1)$ as $\ell\to\infty$, in $L^2$ and
thus in distribution.
It follows that we can let $\ell\to\infty$ in \eqref{vnell},
 see \cite[Theorem 4.2]{Billingsley}, and obtain
\begin{equation}
Z_n\qqw\bigpar{Z_n-m^n\cZ}
\dto \SI\bigpar{1},
\qquad \ntoo.
\end{equation}
This is jointly with all \eqref{si1}, and thus, jointly for all $k\in\bbZ$,
\begin{equation}\label{sw3}
  Z_n\qqw\bigpar{Z_{n-k}-m^{n-k}\cZ}
=Z_n\qqw\bigpar{X_{n,k}+m^{-k}(Z_n-m^n\cZ)}
\dto \SI\bigpar{z^k}.
\end{equation}
Conversely, \eqref{si1} follows immediately from \eqref{sw3}.

In the Galton--Watson case (\refE{EGW}), \eqref{sw3} is 
equivalent to the case $q=0$ of \cite[Theorem (2.10.2)]{Jagers}.
\end{ex}

\subsection{The case $\ggx=m\qqw$}\label{SS2}
Assume now that \refT{T2}  applies; thus $\ggx= m\qqw$ and \eqref{AT3} holds.

In this case, let $\nu$ be the discrete measure , with support $\gGGG$,
\begin{equation}\label{nu2}
  \nu:=
(m-1)\sumpq |1-\gam_p|\qww\,|\mmu'(\gam_p)|\qww\gS(\gam_p) \gd_{\gam_p},
\end{equation}
and consider an isomorphism $I$ of $L^2(\nu)$ into a Gaussian Hilbert space
as above. Then \eqref{t2a}--\eqref{t2b} can be stated as \eqref{si1}, with
the normalizing factor changed from $Z_n\qqw$ to $(nZ_n)\qqw$.

With this change of normalization of $X_{n,k}$, all results in the preceding
subsection hold, with one exception:
The measure $\nu$ has finite support, and thus there exists a sequence
$\ell_j\to\infty$ such that $(zm\qq)^{\ell_j}\to1$ as $j\to\infty$ for every
$z\in\supp(\nu)=\gGGG$;
hence \eqref{eq85} implies
$\limsup_\ltoo\Corr\bigpar{\zeta_k,\zeta_k\llll}=1$.
Hence, 
although the convergence in \eqref{t2b} is mixing, so there is no
dependence on the initial generations as in the case $\ggx<m\qqw$, 
there is a dependence over longer ranges than in the case $\ggx>m\qqw$.

Furthermore, 
 each $\zeta_k$ now belongs to the (typically $q$-dimensional) space spanned
by $\zeta_1,\dots,\zeta_q$, which 
yields the linear dependence of the limits $\zeta_k$ claimed in \refS{Smain}.

\begin{ex}
  In the simplest case, $\gGGG=\set{-m\qq}$. (See \refE{E2} for an example.)
Then $\zeta_k=\bigpar{(-1)^km^{-k/2}-m^{-k}}\zeta$ for some $\zeta\sim
N\bigpar{0,\nu\set{-m\qq}}$ and all $k\in\bbZ$.

Furthermore, $zm\qq=-1$ on
$\supp\nu$, and thus \eqref{esi1a} yields $\zeta_k\llll=(-1)^\ell\zeta_k$;
in particular, $\zeta_k\llll=\zeta_k$ for every even $\ell$.
\end{ex}

\subsection{The case $\ggx<m\qqw$}
In this case, there is no limit, but we can argue with the components of the
approximating sum in \eqref{t3} in the same way as with $\zeta_k$ in
Examples \ref{ESI1}--\ref{ESI2}, and draw the conclusion that \eqref{t3},
interpreted component-wise,
extends also to $k<0$,
as claimed in \refR{Rk<0}.
We omit the details.

\section{Random characteristics}\label{Schar}

A random characteristic is a random function $\gf(t):\ooo\to\bbR$ defined on
the same probability space as the prototype offspring process $\Xi$; we
assume that each individual $x$ has an independent copy $(\Xi_x,\gf_x)$ of
$(\Xi,\gf)$, and interpret $\gf_x(t)$ as the characteristic of $x$ at age $t$.
We consider as above  the lattice case, and 
define, denoting the birth time of $x$ by $\tau_x$,
\begin{equation}\label{zgf}
  \zgf_n:=\sum_{x:\tau_x\le n}\gf_x(n-\tau_x),
\end{equation}
the total characteristic of all individuals at time $n$.
See further \citet{Jagers}.
We assume:
\begin{Cnoenumerate}
\item \label{Agf}
There exists $R_2<m\qq$ such that
$\E[\gf(k)^2]\le C R_2^{2k}$ for some $C<\infty$ and all $k\ge0$.
\end{Cnoenumerate}

We define
\begin{align}
\glgf_k & :=\E \gf(k),
\qquad k\ge0,
\label{glgfk}
\\
\gLgf(z)&:=\sumk \glgf_kz^k,\label{gLgf}
\\
\glgf&:=
\bigpar{1-m\qw}\gLgf\bigpar{m\qw}
=\sumk \bigpar{m^{-k}-m^{-k-1}}\glgf_k, \label{glgf}
\\
\gamb{j,k}&:=
\Cov\bigpar{\gf(j),N_k},
\end{align}
and also  $\glgf_{k}:=0$ for $k<0$. 
Note that \ref{Agf} implies
\begin{equation}\label{glgfb}
  |\glgf_k|=|\E\gf(k)|\le C R_2^k.
\end{equation}
Hence, the sum in \eqref{gLgf} converges absolutely at least for $|z|\le m\qqw$;
in particular,
the sum in \eqref{glgf} converges absolutely. 

We split the characteristic into its mean $\glgf_k=\E\gf(k)$
and the centered part
\begin{equation}
  \chio(k):=\chi(k)-\E\chi(k)=\chi(k)-\glgf_k.
\end{equation}

We define
\begin{equation}\label{vgf}
  \vgf_{n,k}:=
\sum_{x:\tau_x=n}\chio_x(k)=
\sum_{x:\tau_x=n}\bigpar{\gf_x(k)-\glgf_k}
=\sum_{x:\tau_x=n}\gf_x(k)- \glgf_k B_n.
\end{equation}
Then, \eqref{zgf} implies
\begin{equation}
  \zgfo_n=\sumkn \vgf_{n-k,k} 
=\sumk \vgf_{n-k,k}
\end{equation}
and, furthermore,
\begin{equation}
  \zgf_n=\sumkn\bigpar{\vgf_{n-k,k}+\glgf_kB_{n-k}}
=\zgfo_n+\sumk\glgf_kB_{n-k}.
\end{equation}
Hence, recalling \eqref{glgf}, \eqref{bn} and \eqref{ynk},
we have the decomposition
\begin{align}\label{eva}
\zgf_n-\glgf Z_n 
&= \zgfo_n+\sumk\glgf_k\bigpar{B_{n-k}-(m^{-k}-m^{-k-1})Z_n}
\notag\\&
= \zgfo_n+\sumk\glgf_k\bigpar{X_{n,k}-X_{n,k+1}}
\notag\\&
= \zgfo_n +\sum_{k=1}^n\bigpar{\glgf_k-\glgf_{k-1}}X_{n,k}
=\zchia+\zchib,
\end{align}
where $\vecgdl$ is the vector $\bigpar{\glgf_k-\glgf_{k-1}}_{k=0}^\infty$.
Here
$\vecgdl\in \ell^2_{R\qw}$ by \eqref{glgfb}, 
and thus
the asymptotic behaviour of $\zchib$ is given by Theorems \ref{T1}--\ref{T3}.

The term $\zchia$ in \eqref{eva} is asymptotically normal after normalization,
for any value of $\ggx$, as shown by the following theorem.
(Note that the assumption $\E\chi(k)=0$ is equivalent to $\chi=\chio$.)

\begin{thm}  \label{Tgf2}
  Assume \refAA{} and \ref{Agf}. 
If\/ $\E\chi(k)=0$ for every $k\ge0$, then
 as \ntoo,
\begin{equation}\label{tgf2a}
  Z_n\qqw \zgf \dto \zeta^\gf,
\end{equation}
for some  normal random variable $\zeta^\gf$ with mean $\E\zeta^\gf=0$ and
variance 
\begin{equation}\label{tgf2b}
  \begin{split}
\Var\bigpar{\zeta^\gf}
&=
\frac{m-1}m
\sumk m^{-k}\Var\bigpar{\gf(k)}.
  \end{split}
\end{equation}
\end{thm}

Before proving \refT{Tgf2}, we note that
in the case $\ggx>m\qqw$, Theorems \ref{Tgf2} and \ref{T1}
show that $\zchia$ and $\zchib$ in
\eqref{eva} both are asymptotically normal after normalization by $Z_n\qq$. 
In this case, as shown below, 
the two terms are jointly asymptotically normal, leading by \eqref{eva} to
the following extension of \refT{T1} (which is the deterministic case
$\gf(k)=\sum_{j\le k} a_j$).

\begin{thm}  \label{Tgf1}
  Assume \refAAAA. Then, as \ntoo,
\begin{equation}
  Z_n\qqw \bigpar{\zgf-\glgf Z_n} \dto \zeta^\gf,
\end{equation}
for some  normal random variable $\zeta^\gf$ with mean $\E\zeta^\gf=0$ and
variance 
\begin{equation}\label{tgf1}
  \begin{split}
\Var\bigpar{\zeta^\gf}
=
\frac{m-1}m
&\Biggl(
\sumk m^{-k}\Var\bigpar{\gf(k)}
\\
&
-2\oint_{|z|=m\qqw} 
\frac{(1-z)\gLgf(z)-\glgf}{(z-1)(1-\mmu(z))}
\sumk\sumji\gamb{k j} z^j \bar z^k 
\frac{|\dz|}{2\pi m\qqw}
\\
&
+\oint_{|z|=m\qqw}\frac{\lrabs{(1-z)\gLgf(z) -\glgf}^2}
{|1-z|^2\,|1-\mmu(z)|^2}	
\sum_{i,j}\gs_{ij}z^i\bar z^j \frac{|\dz|}{2\pi m\qqw}.
\Biggr)
  \end{split}
\end{equation}
\end{thm}

\begin{rem}
  In both Theorems \ref{Tgf2} and \ref{Tgf1},
joint asymptotic normality for several characteristics, with a corresponding
formula for asymptotic covariances, follow by the proof, or by the
Cram\'er--Wold device.
\end{rem}

\begin{proof}[Proof of Theorems \ref{Tgf2} and \ref{Tgf1}]
We use results from \refS{SpfT1-A}, and assume as we may that $R$ is chosen
with $R_2<R<m\qq$.

Given $B_{n-k}$, $\vgf_{n-k,k}$ is the sum of $B_{n-k}$ independent copies
of $\chio(k) = \gf(k)-\E\gf(k)$. Hence,
using \ref{Agf}, \eqref{EZ} and $B_{n-k}\le Z_{n-k}$,
\begin{equation}
  \E\bigpar{\vgf_{n-k,k}}^2
=\E\bigpar{ \E\bigpar{\vgf_{n-k,k}}^2\mid B_{n-k}}
=\Var\bigpar{\gf(k)} \E B_{n-k}
\le C m^{n-k} R_2^{2k}
\end{equation}
and, using \eqref{glgfb} and \refL{LX1},
\begin{equation}
  \E \bigpar{\glgf_k\xpar{X_{n,k}-X_{n,k+1}}}^2
\le C R_2^{2k} \bigpar{\E X_{n,k}^2 + \E X_{n,k+1}^2}
\le C m^n (R_2/R)^{2k}.
\end{equation}
Since we assume $R_2<R<m\qq$, it follows 
by standard arguments
that if we replace $\gf$ by the truncated
characteristic $\gf_K(k):=\gf(k)\ett{k\le K}$, then the error 
$Z_n\qqw \bigpar{\zgf_n-\glgf Z_n-(Z_n^{\gf_K}-\gl^{\gf_K}Z_n)}$
tends to 0 in probability as $K\to\infty$, uniformly in $n$,
and as a consequence, 
 see \cite[Theorem 4.2]{Billingsley},
it suffices to prove both theorems for the truncated
characteristic $\gf_K$. Hence we may in the sequel assume (changing
notation) that $\gf(k)=0$ for $k>K$, for some $K<\infty$.

Let $\vecgth=(\gth_0,\gth_1,\dots)$ be a random vector such that
$(\vecgth,\veceta)$ is jointly normal with mean $0$ and covariances given
by \eqref{coveta} and
\begin{align}
\Cov(\gth_j,\gth_k)&=\Cov\bigpar{\gf(j),\gf(k)}  ,
\\
\Cov(\gth_j,\eta_k)&=
\gamb{j,k}:=
\Cov\bigpar{\gf(j),N_k}. \label{gamb}
\end{align}
Let $\bigpar{\vecgth\kk,\veceta\kk}$ be independent copies of
$(\vecgth,\veceta)$.

The proof of \refL{Leta} extends to show that \eqref{leta} holds jointly
with
\begin{equation}\label{krk}
  Z_n\qqw \vgf_{n-k,k} \dto \bigpar{1-m\qw}\qq m^{-k/2} \gth\kk_k,
\qquad k\ge0.
\end{equation}

Summing \eqref{krk} over $k\le K$, we obtain
\begin{equation}\label{krk2}
  Z_n\qqw \zgfo_n \dto \zeta^\gf:= \bigpar{1-m\qw}\qq\sumk  m^{-k/2} \gth\kk_k,
\end{equation}
which yields \eqref{tgf2a} and \eqref{tgf2b} in the case $\chi=\chio$;
recall that the terms $\gth\kk_k$ are independent.
This completes the proof of \refT{Tgf2}.

In the remainder of the proof, we thus consider \refT{Tgf1}, and thus assume
that \ref{Aroots} holds. We have just shown that \eqref{leta} holds jointly with
\eqref{krk}.
Hence, by the proof in \refS{SpfT1-A}, \eqref{ele33} holds jointly with
\eqref{krk} for all $k$, and thus also with \eqref{krk2}.
Consequently,
by \eqref{eva}, 
\begin{multline}
  \bigpar{1-m\qw}\qqw Z_n\qqw \bigpar{\zgf_n-\glgf Z_n}
\\
\dto
\sumk m^{-k/2}\gth_k\kk
-\sumk\sumji m^{-(k+j)/2}\eta_j\kkx{k+j} \innprod{T^k(\vec v),\vecgdl}.  
\end{multline}
Write the \rhs{} as $A_1-A_2$, and note that
$A_1$ and $A_2$ are jointly normal with means $0$.
It remains to calculate $\Var(A_1-A_2)$.

Since the terms in the sum $A_1$ are independent,
we have, cf.\ \eqref{krk2} and \eqref{tgf2b},
\begin{equation}\label{kpk}
  \Var(A_1)=\sumk m^{-k}\Var\bigpar{\gth_k}
=\sumk m^{-k}\Var\bigpar{\gf(k)},
\end{equation}
which yields the first term in \eqref{tgf1}, 

$\Var(A_2)$ was
calculated in \refS{SpfT1-A}, see \eqref{ele18} and \eqref{t1b},
which yields the last term in \eqref{tgf1}, 
using $\sum_k (\glgf_k-\glgf_{k-1})z^k = (1-z)\gLgf(z)$ and \eqref{glgf}.

Finally, 
using \eqref{gamb} and \eqref{qkq},
\begin{equation}\label{kqk}
  \begin{split}
&\Cov(A_1,A_2)=
\sumk\sumji m^{-(k+j)}\gamb{k+j,j} \innprod{T^k(\vec v),\vecgdl}
\\&\quad
=
\sumk\sumji  \innprod{T^k(\vec  v),\vecgdl}
\oint_{|z|=m\qqw} z^{k+j}
\suml \bar z^\ell \gamb{\ell,j}
\frac{|\dz|}{2\pi m\qqw}
\\&\quad
=\oint_{|z|=m\qqw} 
 \innprod{(1-zT)\qw(\vec v),\vecgdl}
\suml\sumji z^j \bar z^\ell \gamb{\ell,j}
\frac{|\dz|}{2\pi m\qqw}
\\&\quad
=\oint_{|z|=m\qqw} 
\frac{(1-z)\gLgf(z)-(1-m\qw)\gLgf(m\qw)}{(z-1)(1-\mmu(z))}
\suml\sumji z^j \bar z^\ell \gamb{\ell,j}
\frac{|\dz|}{2\pi m\qqw}
  \end{split}
\end{equation}
The result \eqref{tgf1} follows by combining \eqref{kpk}, \eqref{kqk} and 
\eqref{t1b}, recalling \eqref{glgf}.
\end{proof}

\refT{Tgf1} yields asymptotic normality of $Z_n^\gf$ when $\ggx>m\qqw$,
and \refT{Tgf2} shows the same for any $\ggx$ in the special case
when $\E\chi(k)=0$ for every $k$.
It remains to consider the case when $\glgf_k=\E\chi(k)\neq0$ for some $k$
and $\ggx\le m\qqw$.
If $\ggx=m\qqw$ and \eqref{AT3} holds, 
then \refT{T2} shows that $\zchib/\sqrt{nZ_n}\dto
N(0,\gss)$, where $\gss$ is given  by \eqref{t2b} and
$\gss>0$ except in degenerate cases.
Since \refT{Tgf2} implies that $\zchia/\sqrt{nZ_n}\pto0$, it follows from
\eqref{eva} that $(\zgf_n-\glgf Z_n)/\sqrt{nZ_n}\dto N(0,\gss)$.
Similarly, 
if $\ggx<m\qqw$, then \refT{Tgf2} implies $\ggx^n\zchia\pto0$, and 
\eqref{eva} shows that 
$\zgf_n-\glgf Z_n$ has the same (oscillating) asymptotic behaviour 
as $\zchib$, given by \refT{T3}.

Summarizing, if $\ggx\le m\qqw$, then the randomness in the characteristic
$\chi$ only gives an effect of smaller order than the mean $\E\chi$, and
unless the mean vanishes (or the limits degenerate), 
$\zgf_n$ has the same asymptotic behaviour as if $\chi$ is replaced by the
deterministic $\E\chi$, which is treated by Theorems \ref{T2} and \ref{T3}.

\begin{ex}\label{Edeath}
  We have in the present paper for simplicity assumed \ref{Adeath}, 
that  there are no  deaths. 
Suppose now, more generally, that each individual has a random lifelength
$\ell\le\infty$, as usual with \iid{} copies $(\Xi_x,\ell_x)$ for all
individuals $x$.
The results in \refS{Smain} 
apply if we ignore deaths and let $Z_n$ denote the number
of individuals born up to time $n$, living or dead.
Moreover, the number of living individuals at time $n$ is $\zgf_n$, for the
characteristic $\chi(k):=\ett{\ell> k}$.

Similarly, for example, the number of living individuals at time $n-j$ is
$Z^{\chi_j}_n$ with $\chi_j(k):=\ett{\ell>k-j\ge0}$.
The analogue of $X_{n,j}$ in \eqref{ynk} but counting only living
individuals is thus given by $Z_n^{\chi_j-m^{-j}\chi}$, and results extending
Theorems \ref{T1}--\ref{T3} without assuming \ref{Adeath} follow.
We leave the details to the reader.
\end{ex}

\appendix


\ack

I thank Peter Jagers and Olle Nerman for helpful comments.

%
%
%
%

\newcommand\vol{\textbf}
\newcommand\jour{\emph}
\newcommand\book{\emph}

\end{document}